\documentclass{conm-p-l}
\usepackage{amsbsy}

\usepackage{enumerate}
\usepackage{listings}

\usepackage[pdftex]{graphicx}
\usepackage{epstopdf}

\usepackage{color}
\definecolor{dkgreen}{rgb}{0,0.6,0}
\definecolor{gray}{rgb}{0.5,0.5,0.5}

\usepackage[hyphens]{url}
\usepackage[bookmarks,pagebackref,pdfpagelabels=true]{hyperref}

\hypersetup{
    colorlinks=true,        
    linkcolor=red,          
    citecolor=green,        
    filecolor=magenta,      
    urlcolor=dkgreen           
}

\DeclareMathOperator{\sgn}{sgn}

\date{\today}
\newtheorem{theorem}{Theorem}[section]

\newtheorem{lemma}[theorem]{Lemma}
\newtheorem{remark}[theorem]{Remark}

\numberwithin{equation}{section}

\begin{document}

\title{An Iteratively Reweighted Least Squares Algorithm for Sparse Regularization}


\author{Sergey Voronin}
\address{Mathematics, Tufts University, Medford, MA, 02155, USA.}
\author{Ingrid Daubechies}
\address{Mathematics, Duke University, Durham, NC, 27708, USA.}

\date{\today}

\begin{abstract}
We present a new algorithm and the corresponding convergence analysis for the 
regularization of linear inverse problems with sparsity constraints, applied to a 
new generalized sparsity promoting functional. The algorithm 
is based on the idea of iteratively reweighted least squares, reducing the minimization 
at every iteration step to that of a functional including only $\ell_2$-norms. 
This amounts to smoothing of the absolute value function that appears in the 
generalized sparsity promoting penalty we consider, with the smoothing becoming 
iteratively less pronounced. We demonstrate that the sequence of iterates of our 
algorithm converges to a limit that minimizes the original functional. 
\end{abstract}

\maketitle

\section{Introduction}
Over the last several years, an abundant number of algorithms 
(e.g. \cite{ingrid_thresholding1,cai2009linearized,yin2008bregman,yang2011alternating}) 
have been proposed for the minimization of the $\ell_1$-penalized  
functional $F_{\tau}(x) = \|Ax - b\|_2^2 + 2 \tau \|x\|_1$, 
where the matrix $A$, the vector $x$ and the constant $\tau$
are, respectively, in $\mathbb{R}^{M \times N}$, $\mathbb{R}^N$ and 
$\mathbb{R_{+}}$. This functional has a number of interesting 
applications, such as image restoration \cite{figueiredo2007majorization}, 
face recognition \cite{yang2010fast}, and in inverse problems 
from geophysics \cite{simons2011solving}; one can also view the recovery of corrupted 
low rank matrices as a generalization (since it typically penalizes the $\ell_1$ norm of the 
singular values, i.e. the nuclear norm of the matrix) \cite{wright2009robust}. 
The $\|x\|_1 = \sum_{k=1}^N |x_k|$ penalty is the 
closest norm to the $\ell_0$-penalty (the count of non-zeros in a signal), 
and the relationship between the two has been brought into focus by compressive 
sensing \cite{cande2008introduction}. Since $\|x\|_1$ is not 
differentiable due to the absolute value function $|\cdot|$, standard gradient based 
techniques cannot be directly applied for the minimization of $F_{\tau}$. 
In this paper, we consider a more general functional of which $F_{\tau}$ is a 
particular case. The new functional introduced in \cite{sv_thesis} which the algorithm 
in this paper can minimize is $F_{\bf{q},\pmb{\lambda}}(x)$:
\begin{equation*}
F_{\bf{q},\pmb{\lambda}}(x) = \|Ax - b\|_2^2 + 2 \displaystyle\sum_{k=1}^N \lambda_k |x_k|^{q_k}
\end{equation*} 
where the coefficients $q_k$ and $\lambda_k$ may be different for each $1 \leq k \leq N$, 
with $1 \leq q_k \leq 2$ for each $k$. The more general functional makes it 
possible to treat different components of $x$ differently, corresponding 
to their different roles. A simple example with a half 
sparse, half dense signal is illustrated in the Numerics section; in that case, 
imposing a sparsity inducing penalty on all coefficients is not ideal for proper 
recovery. Another important instance is the case when the penalization 
contains a multiscale representation (e.g. the wavelet decomposition) of an 
object to be reconstructed/approximated. In this case, one has an extra matrix $W$, 
representing the transform to wavelet coefficients, and the minimization problem 
for $w = Wx$ takes the form: 
\begin{equation*}
\bar{w} = \arg\min_w \left\{ \|A W^{-1} w - b\|_2^2 + \displaystyle\sum_{k=1}^N \lambda_k |w_k|^{q_k} \right\}
\end{equation*} 
If $W$ is a wavelet transform, then the different entries of the vector $w$ 
fulfill distinctly different roles, with some being responsible for coarse scales and 
others for fine details. In this case, the total number of possible coefficients 
corresponding to coarse scales is typically quite limited, with each 
of them crucial to the overall model (e.g. \cite{simons2011solving}). 
Thus, we do not necessarily want to impose a sparsity-promoting penalty on these 
coefficients, which means we would choose $q_k > 1$ for 
them in the penalty function. On the other hand, the coefficients corresponding to 
fine scales are typically quite sparse in the object to be reconstructed, and the inversion 
procedure might, without appropriate regularization, be prone to populate them 
with noisy features; in this case, a sparsity promoting choice 
$q_k = 1$ would be indicated for those $k$. 
 
Two approaches are commonly used by various algorithms for the minimization of functionals that,
like $F_{\bf{q},\pmb{\lambda}}$, involve a non-smooth absolute value term. 
The first approach handles the non-smooth minimization problem directly. 
For instance, for our original $F_{\tau}$, one uses the soft-thresholding 
operation \cite{ingrid_thresholding1} on $\mathbb{R}$, defined by: 
\begin{equation*}\label{Sthr} S_{\tau }(x) = \left\{%
\begin{array}{ll}
    x - \tau, & \hbox{$x\ge \tau$;} \\
    0, & \hbox{$-\tau \le x \le \tau$;} \\
    x + \tau, & \hbox{$x \le -\tau$ .} \\
\end{array}%
\right.\end{equation*}
For a vector of $N$ elements, soft-thresholding is then defined component-wise by setting,
$\forall\,k = 1, \dots, N$, 
$(\mathbb{S}_{\tau}(x))_k = S_{\tau} (x_k)$.
The use of soft-thresholding relies on the identity  
$S_{\tau}(\beta) = \arg\min_a \left\{ (a - \beta)^2 + 2\tau |a| \right\}$ for scalars $a$ and 
$\beta$, which for vectors $x$ and $b$ translates to:
\begin{equation}
\label{eq:soft_thresholding_min_def}
\mathbb{S}_{\tau}(b) = \arg\min_x \left\{ \|x - b\|_2^2 + 2\tau\|x\|_1 \right\}
\end{equation}
The simplest example is the Iterative Soft Thresholding Algorithm (ISTA) 
\cite{ingrid_thresholding1}:
\begin{equation}
\label{eq:ista_scheme}
x^{n+1} = \mathbb{S}_{\tau}(x^n + A^T b - A^T A x^n)
\end{equation}
which for an initial $x^0$ and with $\|A\|_2 < 1$ (easily accomplished by rescaling;
$\|A\|_2$ is the operator norm of $A$ from $\ell_2$ to $\ell_2$, also called
the spectral norm of $A$), converges slowly but surely 
to the $\ell_1$-minimizer. A faster variation on this scheme, known as FISTA 
\cite{Beck2009}, is frequently employed; the thresholding function 
can also be adjusted to correspond to more general penalties \cite{ICASSP_paper}. 
Along the same line of thinking, algorithms based on the dual space of the 
$\ell_1$-norm have been proposed \cite{yang2011alternating}, 
with the dual being the $\ell_{\infty}$-norm.

The second approach to algorithms minimizing the $\ell_1$-based functional involves 
some kind of smoothing. One idea is to replace the entire functional by a smooth 
approximation. This can be done, for instance, by convolving the absolute 
value function with narrow Gaussians \cite{2014arXiv1408.6795V}. 
This approach then allows for the use of standard gradient based methods 
(such as Conjugate Gradients) for the 
minimization of the approximate smooth functional. The main problem 
with this approach is that we are then minimizing a slightly different 
functional from the original that does not necessarily have the same 
properties that the original penalty possesses. 

The algorithm described in this paper replaces the $|x_k|^{q_k}$ term 
in $F_{\bf{q},\pmb{\lambda}}$ with a smoothened version that tends to the 
original as the iterates progress towards the limit. 
This algorithm builds upon the original iteratively 
reweighted least squares (or IRLS) method proposed in \cite{daubechies2010iteratively} (as 
well as earlier work in \cite{figueiredo2007majorization,fuchs2007convergence}), 
extending it to the unconstrained case and to a more general penalty.
The idea can be illustrated simplest for the $q_k = 1$ case. 
Consider the approximation:
\begin{equation*}
|x_k| = \frac{x_k^2}{|x_k|} = \frac{x_k^2}{\sqrt{x_k^2}} \approx \frac{x_k^2}{\sqrt{x_k^2 + \epsilon^2}}
\end{equation*}
where in the rightmost term, a small $\epsilon \neq 0$ is used, 
to insure the denominator is finite, regardless of the value of $x_k$. 
Thus, at the $n$-th iteration, a reweighted $\ell_2$-approximation to the 
$\ell_1$-norm of $x$ is of the form:
\begin{equation*}
\|x\|_1 \approx \displaystyle\sum_{k=1}^N \frac{x_k^2}{\sqrt{(x^n_k)^2 + \epsilon_n^2}} = \displaystyle\sum_{k=1}^N \tilde{w}^n_k x_k^2
\end{equation*}
where the right hand side is a reweighted two-norm with weights:
\begin{equation}
\label{eq:irls_scheme_weights_ell1}
\tilde{w}^n_k = \frac{1}{\sqrt{(x^n_k)^2 + \epsilon_n^2}}.
\end{equation}
Clearly, it follows that $\sum_k \tilde{w}^n_k (x^n_k)^2$ is a close approximation
to $\|x^n\|_1$. In the same way, we can use the slightly more general weights:
\begin{equation}
\label{eq:irls_scheme_weights}
w^n_k = \frac{1}{\left[(x^n_k)^2 + \epsilon_n^2\right]^\frac{2-q_k}{2}}.
\end{equation}
for the approximation $|x^n_k|^{q_k} \approx w^n_k (x^n_k)^2$ to hold; 
these can then deal with the case $1 \leq q_k \leq 2$.

We shall use a sequence $\{\epsilon_n\}$ such that $\epsilon_n \to 0$ as 
$n \to \infty$. We note that it is important that $\epsilon_n > 0$ for all $n$ for a 
rigorous convergence proof. In an approach where $\epsilon_n = 0$, entries $k$ for 
which $x^n_k = 0$ would lead to diverging $\tilde{w}^n_k$ and hence to 
$x_k^{n^\prime} = 0$ for all subsequent $n^{\prime} > n$. This is OK if the $k$-th entry 
of the minimizer is indeed zero; if (as is typically the case) this cannot be guaranteed, 
convergence would fail. 

A precise choice of the sequence 
$\{\epsilon_n\}$ is important for convergence analysis. Although in practice, 
different approaches can work, the rate at which the $\epsilon$-sequence converges needs to match that 
of the iterates $x^n$. In our analysis, we will use the following definition:
\begin{equation}
\label{eq:irls_scheme_epsilons}
\epsilon_{n} = \min\left(\epsilon_{n-1},\left( \|x^n - x^{n-1}\|_2 + \alpha^n \right)^{\frac{1}{2}}\right), 
\end{equation}
where $\alpha \in (0,1)$ is some fixed number. The resulting algorithm we present and 
analyze is very similar in form to \eqref{eq:ista_scheme}:
\begin{equation}
\label{eq:irls_scheme}
x^{n+1}_k = \frac{1}{1 + \lambda_k q_k w^n_k} \left(x^n_k + (A^T b)_k - (A^T A x^n)_k\right) \quad \mbox{for} \quad k=1,\dots,N,
\end{equation}
with the thresholding replaced by an iteration-dependent scaling operation using 
the weights \eqref{eq:irls_scheme_weights}. The algorithm is found to be numerically 
competitive with the thresholding based schemes for the $\ell_1$-case but has the 
advantage that it can handle the minimization of more general functionals of the 
form $F_{\bf{q},\pmb{\lambda}}$. The main contribution of 
this paper is a detailed proof of convergence, the methodology of which can be readily 
applied to analyze similar schemes. An added advantage of a scheme in which all terms 
are quadratic in the unknown $x$ is that it can be combined with a conjugate gradient 
approach to speed up the algorithm. In \cite{sv_thesis} such an algorithm was 
proposed, and convergence proved if at each reweighted step, the conjugate gradient scheme 
was pursued to convergence. In \cite{2015arXiv150904063F}, the more general and 
more realistic situation is considered, where only some conjugate gradient steps 
are taken at each iteration. In both cases, the choice of $\{ \epsilon_n \}$ 
(e.g. \eqref{eq:irls_scheme_epsilons}), remains crucial for the convergence 
analysis.

\section{Constructions}

\subsection{Analysis of the generalized sparsity inducing functional}
Here, we derive and comment on the optimality conditions of the 
functional:
\begin{equation}
\label{eq:ellq_funct}
F(x) = \|Ax - b\|_2^2 + 2 \displaystyle\sum_{k=1}^N \lambda_k |x_k|^{q_k},
\end{equation} 
for the range $1 \leq q_k \leq 2$, where in \eqref{eq:ellq_funct}, we drop the 
subscripts $\bf{q}$ and $\pmb{\lambda}$ for convenience. 
Notice that since \eqref{eq:ellq_funct} 
is convex for the range of $q_k$ specified, every local minimizer is a 
global minimizer of the functional. The optimality conditions for a general vector $x$ with 
components $x_k$ for $k \in (1, \dots, N)$ can be written down in component-wise form, 
as derived in Lemma \ref{lem:minFellqcharacter} 
below. Note that as $F_{\bf{1},\lambda}$ is a special case of 
$F_{\bf{q},\pmb{\lambda}}$, the component-wise conditions below 
reduce to the well known optimality conditions of the $\ell_1$ 
penalized functional when $q_k = 1$ for all $k$.   
\begin{lemma}
\label{lem:minFellqcharacter}
The conditions for the minimizer of the functional $F(x)$ as defined 
in \eqref{eq:ellq_funct} are:
\begin{equation}
\label{eq:optcond}
\begin{array}{rll}
\{A^T (b - Ax)\}_k &= \lambda_k \sgn(x_k) q_k |x_k|^{q_k-1} \,,& x_k \neq 0 \quad (1 \leq q_k \leq 2) \\                         
\{A^T (b - Ax)\}_k &= 0 \,,& x_k = 0 \quad (q_k > 1) \\                                               
\left|\{A^T (b - Ax)\}_k\right| &\leq \lambda_k \,,& x_k = 0 \quad (q_k = 1)                           
\end{array}
\end{equation}
\end{lemma}
\begin{proof}
Since for the case $1 \leq q_k \leq 2$, $F(x)$ is convex, any local minimizer is necessarily
global. Thus, to characterize the minimizer, it is necessary only to work 
out the conditions corresponding to 
$F(x) \leq F(x + tz)$ for all sufficiently small $t \in \mathbb{R}$ and all $z\in\mathbb{R}^N$. 
$F(x) \leq F(x + tz)$ implies that: 
\begin{equation}
\label{eq:optcondstep1}
t^2 ||A z||^2  + 2t\langle z,A^T (Ax -b) \rangle + 2\displaystyle\sum_{k=1}^N \lambda_k \left( |x_k + t z_k|^{q_k} - |x_k|^{q_k} \right) \geq 0.
\end{equation}
We derive $N$ conditions, one for each index $k\in\{1,\ldots,N\}$; for the $k$-th condition, we consider $z$ of the special form $z=z_k e_k$ (i.e.~ only the $k$-th entry of $z$ differs from 0).
We separately analyze the cases $x_k \neq 0$ and $x_k = 0$, starting with the former.

When $x_k \neq 0$, the function $f(t) = |x_k + tz_k|^{q_k}$ 
is $C^{\infty}$ at $t=0$. Using a Taylor series expansion 
around $0$, we then get $f(t) = f(0) + t f'(0) + O(t^2)$. In addition, 
$\sgn(x_k + t z_k) = \sgn(x_k)$ for sufficiently small $t$. Keeping $t$ fixed we 
analyze both signs of $x_k$. 
For $x_k > 0$, we have $\sgn(x_k) = 1$ and $|x_k + t z_k| = x_k + t z_k$, so that:
\begin{equation*}
f(t) = (x_k + t z_k)^{q_k} \implies f^{\prime}(t) = \sgn(x_k) q_k z_k (x_k + t z_k)^{q_k - 1} = \sgn(x_k) q_k z_k |x_k + t z_k|^{q_k - 1}.
\end{equation*}
When $x_k < 0$, we have $\sgn(x_k) = -1$ and $|x_k + t z_k| = -(x_k + t z_k)$, so that:
\begin{equation*}
f(t) = (-x_k - t z_k)^{q_k} \implies f^{\prime}(t) = - q_k z_k (-x_k - t z_k)^{q_k - 1} = 
\sgn(x_k) q_k z_k |x_k + t z_k|^{q_k - 1}.
\end{equation*}
Thus, $f^{\prime}(0) = \sgn(x_k) q_k z_k |x_k|^{q_k-1}$ for all $x_k \neq 0$.
Thus, there exists a constant $C>0$ such that the Taylor expansion of $f$ becomes:
\begin{eqnarray*}
f(t) 
&=& |x_k + tz_k|^{q_k} 
= |x_k|^{q_k} + t \sgn(x_k) q_k z_k |x_k|^{q_k-1} + O(t^2) \\
&\leq& |x_k|^{q_k} + t \sgn(x_k) q_k z_k |x_k|^{q_k-1} + Ct^2,
\end{eqnarray*}
This implies in particular that 
$|x_k + t z_k|^{q_k} - |x_k|^{q_k} \leq t \sgn(x_k) q_k z_k |x_k|^{q_k-1} + Ct^2$. 
Using this and $z= z_k e_k$ in \eqref{eq:optcondstep1} gives:
\begin{eqnarray*}
t^2 \|A (z_k e_k)\|^2  + 2t\langle z_k e_k,A^T (Ax -b)\rangle + 2 \lambda_k \left( t \sgn(x_k) q_k z_k |x_k|^{q_k-1} + C t^2 \right) \geq 0 \\
\implies \ t^2 \left( \|A (z_k e_k)\|^2 + 2 C \lambda_k \right) + 2t \left( z_k \{A^T (Ax - b)\}_k + \lambda_k \sgn(x_k) q_k z_k |x_k|^{q_k-1} \right) \geq 0 .
\end{eqnarray*}
The first term can be made arbitrary small with respect to the second; the inequality will thus hold for both $t>0$ and $t<0$ iff:
\begin{equation*}
z_k \{A^T (Ax - b)\}_k + \lambda_k \sgn(x_k) q_k z_k |x_k|^{q_k-1} = 0,
\end{equation*}
which leads to:
\begin{equation*}
\{A^T (b - Ax)\}_k = \lambda_k \sgn(x_k) q_k |x_k|^{q_k-1} \,,\  x_k \neq 0.
\end{equation*}
Note that when $q_k = 1$ we recover the familiar condition for minimization of the $\ell_1$-functional:
\begin{equation*}
\{A^T (b - Ax)\}_k = \lambda_k \sgn(x_k) \,,\ x_k \neq 0.
\end{equation*}

\vspace{2.mm}
When $x_k = 0$, recalling that $z= z_k e_k$, \eqref{eq:optcondstep1} gives:
\begin{equation}
\label{eq:optcondstep2}
t^2 ||A (z_k e_k)||^2  + 2t\langle z_k e_k,A^T (Ax -b)\rangle + 2 \lambda_k |t|^{q_k} |z_k|^{q_k} \geq 0.
\end{equation}
Making the substitutions $t^2 = |t|^2$, $t = |t|\sgn(t)$, 
we obtain
\begin{equation}
\label{eq:optcondstep3}
|t|^2 ||A (z_k e_k)||^2  + |t|\left(2\sgn(t) z_k \{A^T (Ax -b)\}_k + 2 \lambda_k |t|^{q_k-1} |z_k|^{q_k} \right) \geq 0. \end{equation}
In this case, we have to consider the case $q_k = 1$ and $q_k > 1$ separately. When $q_k > 1$ we have that:
\begin{equation*}
|t|^2 ||A z||^2  +  2\lambda_k |t|^{q_k} |z_k|^{q_k} + 2 |t| \sgn(t)z_k \{A^T (Ax -b)\}_k  \geq 0.
\end{equation*}
Since $q_k > 1$, the first two terms on the left have greater powers of $|t|$ than the last term and can be made 
arbitrarily smaller by picking $t$ small enough. This means we must have:
\begin{equation*}
2 \sgn(t) z_k \{A^T (Ax - b)\}_k \geq 0
\end{equation*}
for all $t$, which can be true only if $\{A^T (Ax - b)\}_k = 0$. Thus, we conclude that 
the condition is:
\begin{equation*}
\{A^T (b - Ax)\}_k = 0 \,,\ x_k = 0 \quad (q_k > 1).
\end{equation*}
For $q_k = 1$, applying a similar argument to \eqref{eq:optcondstep3} leads to:
\begin{equation*}
\begin{array}{rrl}
&|t|^2 ||A z||^2  + |t|\left( 2\sgn(t) z_k \{A^T (Ax -b)\}_k + 2\lambda_k |z_k| \right)& \geq 0
\\
\implies& \sgn(t) z_k \{A^T (Ax -b)\}_k + \lambda_k |z_k|& \geq 0 .
\end{array}
\end{equation*}
Now consider the two cases: where $t$ has the same sign as $z_k$,
$\sgn(t)=\sgn(z_k)$ or the opposite sign, $\sgn(t)=-\sgn(z_k)$. They lead to, respectively:
\begin{equation*}
\{A^T (Ax - b) \}_k + \lambda_k \geq 0 \quad \mbox{and} \quad 
-\{A^T (Ax - b) \}_k + \lambda_k \geq 0,
\end{equation*}
so we obtain the condition:
\begin{equation*}
\bigl|\{A^T (b - Ax)\}_k\bigr| \leq \lambda_k \,,\  x_k = 0 \quad (q_k = 1).
\end{equation*}
\vspace{2.mm}
Thus, we can summarize the component-wise conditions for the minimizer of $F(x)$ as in \eqref{eq:optcond}.
\end{proof}

The conditions derived in Lemma \ref{lem:minFellqcharacter} allow us to pick a 
strategy for selecting $\{\lambda_k\}$. As an example, 
for the case $q_k = 1$ and $\lambda_k = \lambda$ for all $k$ 
we have that for $\lambda > \|A^T b\|_{\infty}$, the optimal 
solution is the zero vector. Hence, we typically would start at some 
value of $\lambda$ just below $ \|A^T b\|_{\infty}$ where the zero vector is a good 
initial guess. We can then iteratively decrease $\lambda$ and use the previous 
solution as the initial guess at the next lower $\lambda$ while we go down 
to some target residual. Well-known techniques such as the 
L-curve method \cite{hansen1999curve} apply here.  

\subsection{Derivation of the algorithm}

The iteratively reweighted least squares (IRLS) algorithm given by scheme 
\eqref{eq:irls_scheme} with weights \eqref{eq:irls_scheme_weights} follows from 
the construction of a surrogate functional \eqref{eq:Gdef} which we will use in our 
analysis, as presented in Lemma \ref{lem:irls_surrogate_functional} below.
In our constructions, we split the index set $1 \leq k \leq N$ into two 
parts: $Q_1 = \{k : 1 \leq q_k < 2\}$ and $Q_2 = \{k : q_k = 2\}$.

\begin{lemma}
\label{lem:irls_surrogate_functional}
Define the surrogate functional:
\begin{eqnarray}
\nonumber
\quad
G(x,a,w,\epsilon) &=& \|Ax - b\|_2^2 - \|A(x - a)\|_2^2 + \|x - a\|_2^2 \\
\label{eq:Gdef}
&+& \displaystyle\sum_{k \in Q_1} \lambda_k \left( q_k w_k \left((x_k)^2 + \epsilon^2\right) + (2- q_k) (w_k)^{\frac{q_k}{q_k - 2}} \right) \\
\nonumber
&+& \displaystyle\sum_{k \in Q_2} \left[ 2 \lambda_k \left( (x_k)^2 + \epsilon^2 \right) (w_k^2 - 2 w_k + 2) \right].
\end{eqnarray}
Then the minimization procedure
$~w^{n} = \arg\min_w G(x^{n},a,w,\epsilon_{n})~$
defines the iteration dependent weights:
\begin{equation}
\label{eq:witer}
w^{n}_k = \frac{1}{\left[(x^{n}_k)^2 + (\epsilon_{n})^2\right]^\frac{2-q_k}{2}}.
\end{equation}
In addition, the minimization procedure
$~x^{n+1} = \arg\min_x G(x,x^n,w^n,\epsilon_n)~$
produces the iterative scheme:
\begin{equation}
\label{eq:xiter}
x^{n+1}_k = \frac{1}{1 + \lambda_k q_k w^n_k} \left( (x^n)_k - (A^T A x^n)_k +  (A^T b)_k \right).
\end{equation}
\end{lemma}

\begin{proof}
For the derivation of the weights from 
$w^{n} = \arg\min_w G(x^{n},a,w,\epsilon_{n})$, 
we take only the terms of $G$ that depend on $w$. We derive separately 
the weights for $k \in Q_1$ and $k \in Q_2$. First, for $k \in Q_1$: 
\begin{equation*}
\begin{array}{rrl}
&& \displaystyle 
\frac{\partial}{\partial w_k} \left[ q_k w_k ((x^{n}_k)^2 + (\epsilon_{n})^2) + (2 - q_k) (w_k)^{\frac{q_k}{q_k-2}} \right] =\ 
\displaystyle 0 
\\
& \implies&\displaystyle 
 q_k \left((x^{n}_k)^2 + (\epsilon_{n})^2\right) + (2-q_k) \frac{q_k}{q_k - 2} (w_k)^{\frac{q_k}{q_k-2} -1} =\ 
\displaystyle 0 
\\
& \implies& w^{n}_k =
\displaystyle \frac{1}{\left[(x^{n}_k)^2+(\epsilon_{n})^2\right]^\frac{2-q_k}{2}} .
\end{array}
\end{equation*} 
Next, for $k \in Q_2$, we have:
\begin{eqnarray*}
&& \frac{\partial}{\partial w_k} \left[ 
2 \lambda_k \left( (x^{n}_k)^2 + (\epsilon_n)^2 \right) (w_k^2 - 2 w_k + 2) \right] = 
2 \lambda_k \left( (x^{n}_k)^2 + (\epsilon_n)^2 \right) (2 w_k - 2) = 0 \\
&& \implies w^n_k = 1
\end{eqnarray*}
Notice that this implies that \eqref{eq:witer} is valid for $k$ in both sets 
$Q_1$ and $Q_2$ since for $k \in Q_2$, $q_k = 2$ and 
\eqref{eq:witer} gives $w^n_k = 1$ as required. 

\vspace{2.mm}
Next, we verify that the definition:
\begin{equation*}
x^{n+1}_k = \left\{\arg\min_x G(x,x^n,w^n,\epsilon_n) \right\}_k
\end{equation*}
recovers the iterative scheme \eqref{eq:xiter}. Using that $w^n_k = 1$ for 
$k \in Q_2$, as just derived, we have:
\begin{eqnarray}
\label{eq:Gatxn}
\nonumber
G(x,x^n,w^n,\epsilon_n) &=& \|Ax - b\|_2^2 - \|A(x - x^n)\|_2^2  + \|x - x^n\|_2^2\\
&& \quad + \sum_{k \in Q_1} \lambda_k \left( q_k w^n_k ((x_k)^2 + (\epsilon_n)^2) + (2-q_k) (w^n_k)^\frac{q_k}{q_k - 2} \right)\\
\nonumber
&& \quad + \sum_{k \in Q_2} \left[ 2 \lambda_k \left( (x_k)^2 + (\epsilon_n)^2 \right) \right].
\end{eqnarray}
To prove \eqref{eq:xiter}, we again separately analyze the 
cases $k \in Q_1$ and $k \in Q_2$. 
We differentiate \eqref{eq:Gatxn} with respect to $x$, then take the $k$-th component 
and set to zero. For $k \in Q_1$, removing terms of \eqref{eq:Gatxn} that do not depend 
on $x$, we get:
\begin{eqnarray*}
&& \frac{\partial}{\partial x_k} \left( \|Ax - b\|_2^2 - \|A(x - x^n)\|_2^2 + 
\|x - x^n\|_2^2 +  \displaystyle\sum_{k \in Q_1} \lambda_l q_l w^n_l x^2_l \right) \\
&=& \frac{\partial}{\partial x_k} \left( \|x\|_2^2 - 2\left(x, x^n + A^T b - A^T A x^n \right)  +  \displaystyle\sum_{k \in Q_1} \lambda_l q_l w^n_l x^2_l \right) = 0
\end{eqnarray*}
and the result is:
\begin{equation*}
-2\{A^T b\}_k + 2\{A^TA x^n\}_k + 2x_k - 2x^n_k + 2 \lambda_k q_k w_k^n x_k = 0.
\end{equation*}
Then we solve for $x_k$ and define $x^{n+1}_k$ to be the result:
\begin{eqnarray*}
x_k (1 + \lambda_k q_k w^n_k) &=& x^n_k + \{A^Tb\}_k - \{A^T A x^n\}_k \\
\implies \qquad\qquad\quad 
x_k^{n+1} &=& \frac{1}{1 + \lambda_k q_k w^n_k} \left\{ x^n + A^Tb - A^T A x^n \right\}_k.
\end{eqnarray*}
For $k \in Q_2$, $w^n_k = 1$ and we obtain:
\begin{eqnarray*}
&& \frac{\partial}{\partial x_k} \left( \|x\|_2^2 - 2\left(x, x^n + A^T b - A^T A x^n \right)  +  \displaystyle\sum_{k \in Q_2} 2 \lambda_k x_k^2 \right) \\
&&= -2\{A^T b\}_k + 2\{A^TA x^n\}_k + 2x_k - 2x^n_k + 4 \lambda_k x_k = 0.
\end{eqnarray*}
which, upon solving for $x_k$, yields the scheme:
\begin{equation*}
x_k^{n+1} = \frac{1}{1 + 2 \lambda_k} \left\{ x^n + A^Tb - A^T A x^n \right\}_k.
\end{equation*}
Thus, it follows that \eqref{eq:xiter} holds for all $1 \leq k \leq N$.
\end{proof}

\begin{remark}
\label{rem:surrogate_and_functional}
Assume that as $n \rightarrow \infty$, $x^n \rightarrow x$ and $\epsilon_n \rightarrow 0$.
Notice that with the weights in \eqref{eq:witer}, we have that:  
\begin{equation*}
w^n_k (x^n_k)^2 
= \frac{(x^n_k)^2}{\left((x_k^n)^2 + 
(\epsilon_n)^2\right)^\frac{2-q_k}{2}} 
\to \frac{x_k^2}{\left(x_k^2 + 0 \right)^\frac{2-q_k}{2}}  = |x_k|^{q_k} \ \mbox{ as }\ n\to\infty, \mbox{   if   } x_k \neq 0.
\end{equation*}
Next, observe the result of the computation:
\begin{eqnarray}
\label{eq:Gwnxnsimplification}
\nonumber
&&  q_k w^n_k \left((x_k)^2 + (\epsilon_n)^2\right) + (2-q_k) (w^n_k)^\frac{q_k}{q_k - 2} \\ 
& =& q_k \left((x^{n}_k)^2 + (\epsilon_n)^2\right)^{\left(\frac{q_k - 2}{2} + \frac{2}{2}\right)} + (2-q_k) \left((x^{n}_k)^2 + (\epsilon_n)^2\right)^{\left( \frac{q_k - 2}{2} \frac{q_k}{q_k - 2} \right)} \\
\nonumber
& =& 2\left((x^n_k)^2 +(\epsilon_n)^2\right)^{\frac{q_k}{2}}.
\end{eqnarray}
It follows from \eqref{eq:Gatxn} and $q_k = 2$, $w^n_k = 1$ for $k \in Q_2$ that:
\begin{eqnarray}
\label{eq:Gatxn2}
\nonumber
G(x^n,x^n,w^n,\epsilon_n)= \|Ax^n - b\|_2^2 &+& \sum_{k \in Q_1} \lambda_k \left( q_k w^n_k ((x_k)^2 + (\epsilon_n)^2) + (2-q_k) (w^n_k)^\frac{q_k}{q_k - 2} \right)\\
\nonumber
&+& \sum_{k \in Q_2} \left[ 2 \lambda_k \left( (x_k)^2 + (\epsilon_n)^2 \right) \right],
\end{eqnarray}
which using \eqref{eq:Gwnxnsimplification}, reduces to:
\begin{eqnarray}
\|Ax^n - b\|_2^2 + 2 \displaystyle\sum_{k \in Q_1} \lambda_k \left((x^n_k)^2 + (\epsilon_n)^2\right)^\frac{q_k}{2} + 2 \displaystyle\sum_{k \in Q_2} \lambda_k \left((x^n_k)^2 + (\epsilon_n)^2\right)^\frac{2}{2}.
\end{eqnarray}
Thus, we recover:
\begin{equation}
\label{eq:gxnxnwnen}
G(x^n,x^n,w^n,\epsilon_n) = \|Ax^n - b\|_2^2 + 2 \displaystyle\sum_{k=1}^N \lambda_k \left((x^n_k)^2 + (\epsilon_n)^2\right)^\frac{q_k}{2},
\end{equation}
As $n \rightarrow \infty$, assuming $x^n \rightarrow x$ and $\epsilon_n \rightarrow 0$, 
we have that:
\begin{equation*}
\lim_{n \to \infty} G(x^n,x^n,w^n,\epsilon_n) = \|Ax - b\|_2^2 + 2 \displaystyle\sum_{k=1}^N \lambda_k |x_k|^{q_k},
\end{equation*}
so we recover the functional \eqref{eq:ellq_funct} we would like to minimize.
\end{remark}

\subsection{Summary of argument flow}
\label{subsect:argflow}
\textit{Notation}: With some abuse of notation, we will denote by $\{a_n\}$ the 
sequence $(a_n)_{n \in \mathbb{N}}$, and write $\{a_{n_l}\}$, $\{a_{n_{l_r}}\}$ 
for subsequences $(a_{n_l})_{l \in \mathbb{N}}$, 
$(a_{n_{l_r}})_{r \in \mathbb{N}}$, respectively. By $F$ we will refer to the functional 
$F_{\bf{q},\pmb{\lambda}}(x)$ in \eqref{eq:ellq_funct}. We demonstrate that 
for our set of iterates $\{x^n\}$ from \eqref{eq:irls_scheme}, we have 
convergence to the minimizing value, i.e. $\lim_{n \to \infty} F(x^n) = F(\bar{x})$, 
where $\bar{x}$ is such that $F(\bar{x}) \leq F(x)$ for all $x$. Under some 
conditions on $F$, the minimizer will be unique. In that case, we have 
that $x^n \to \bar{x}$. These statements will all follow from a few properties of 
$F$ and $G$ (from \eqref{eq:ellq_funct} and \eqref{eq:Gdef}) and the sequence of 
iterates $x^n$ from \eqref{eq:irls_scheme}, 
which we now state, and which will be proved in Section 3:

\begin{enumerate}[(1)]
\item $0 \leq F(x^n) \leq G(x^n,x^n,w^n,\epsilon_n), \ \forall \, n$.
\item $G(x^{n},x^{n},w^{n},\epsilon_{n}) \leq G(x^{n-1},x^{n-1},w^{n-1},\epsilon_{n-1}), \ \forall \, n$.
\item $\exists$ subsequence $\{x^{n_l}\}$ of $\{x^n\}$ for 
which $\lim_{l \to \infty} \left[ G(x^{n_l},x^{n_l},w^{n_l},\epsilon_{n_l}) - F(x^{n_l}) \right] = 0$.
\item $\|x^n\|$ is bounded, which implies that any subsequence 
of $\{x^n\}$ has a weakly convergent subsequence; in particular $\{x^{n_l}\}$ 
has a convergent subsequence $\{x^{n_{l_r}}\}$. 
\item The limit $\bar{x}$ of the particular convergent subsequence 
$\{x^{n_{l_r}}\}$ satisfies the optimality conditions of $F$ 
(i.e. $F(\bar{x}) \leq F(x)$ for all $x$).
\end{enumerate}

We now show that these statements suffice to conclude that 
$\lim_{n \to \infty} F(x^n) = F(\bar{x})$, 
an important result, as it states that the iterates converge to the minimizing 
value of the functional. First, let us define 
the sequence $\{g_n\} := G(x^n,x^n,w^n,\epsilon_n)$. Note from (1) and (2) 
that $\{g_n\}$ is bounded from below and monotonically decreasing; it follows 
that this sequence converges as $n \to \infty$, 
say to some $\bar{g}$. Consequently, 
$\{ G(x^{n_l}, x^{n_l}, w^{n_l}, \epsilon_{n_l}) \} = \{ g_{n_l} \}$ converges to 
$\bar{g}$ as $l \to \infty$. By (3) it then follows that $\{ F(x^{n_l}) \}$ 
also converges to $\bar{g}$ as $l \to \infty$. 
Since we know that  
$x^{n_{l_r}} \to \bar{x}$, it follows from the continuity of $F$ that 
$\{ F(x^{n_{l_r}}) \} \to \{ F(\bar{x}) \}$; consequently $\bar{g} = F(\bar{x})$ and 
hence $F(x^{n_l}) \to F(\bar{x})$ as $l \to \infty$, where 
$F(\bar{x}) \leq F(x)$ for all $x$.

Finally, we like to show that $F(x^n) \to F(\bar{x})$. Note that for any 
$\sigma > 0$, $\exists L$ such that $\forall l \geq L$ we have 
that $|F(x^{n_l}) - F(\bar{x})| < \sigma$. Next, for every $n \geq n_l \geq l$, 
we have that:
\begin{equation*}
F(x^{n_l}) = g_{n_l} \geq g_n = G(x^n, x^n, w^n, \epsilon_n) \geq F(x^n) 
\end{equation*}
where $g_{n_l} \geq g_n$ since $n_l \leq n$. 
So this means that $F(x^{n_l}) \geq F(x^n)$ and we know from before that 
$|F(x^{n_l}) - F(\bar{x})| =  F(x^{n_l}) - F(\bar{x}) < \sigma$, 
which implies that $F(x^n) - F(\bar{x}) < \sigma$ for $n \geq n_l$, where 
we have used that $F(\bar{x}) \leq F(x)$ for all $x$.
It follows that $F(x^n) \to F(\bar{x})$. This implies, in particular, that 
for any accumulation point $\hat{x}$ of $\{x^n\}$, we have $F(\hat{x}) = F(\bar{x})$ 
(since $\hat{x}$ is the limit of a subsequence of $\{x^n\}$ and $F$ is continuous).
In the case that the minimizer of $F$ is unique and equal to $\bar{x}$, it follows 
that $\bar{x}$ is the only possible accumulation point of $\{x^n\}$, i.e. 
that $x^n \to \bar{x}$.
The majority of the work in the convergence argument which follows goes 
into introducing a proper construction for the $\{\epsilon_n\}$ sequence and 
showing that the properties (1) - (5) hold for this choice. 

\section{Analysis of the IRLS algorithm}

Having set out the fundamentals (derivation of the scheme and outline of the 
convergence proof), we now analyze the IRLS scheme 
in \eqref{eq:irls_scheme}, with weights $w^n_k$ defined by 
\eqref{eq:irls_scheme_weights} and $\{\epsilon_n\}$ as defined by 
\eqref{eq:irls_scheme_epsilons}; we establish convergence by proving properties 
(1) to (5) from Section \ref{subsect:argflow}. We will 
assume that $\|A\|_2 < 1$. (I.e., $A$ has spectral or operator norm, or equivalently 
largest singular value, less than 1, 
which can be accomplished by simple rescaling. 
The largest singular value can typically be estimated accurately using a few 
iterations of the power scheme.)  

\begin{lemma}
Let the surrogate functional $G$ be given by  \eqref{eq:Gdef}
of {\rm Lemma \ref{lem:irls_surrogate_functional}} and $F$ be the functional 
in \eqref{eq:ellq_funct}. Then property {\rm (1)} above holds.
\end{lemma}
\begin{proof}
The proof follows by direct verification using the 
result of Remark \ref{rem:surrogate_and_functional}.
\[
G(x^n,x^n,w^n,\epsilon_n) 
=
\|A x^n - b\|_2^2 + 
2\displaystyle\sum_{k=1}^N \lambda_k \left( (x^n_k)^2 + (\epsilon_n)^2 \right)^{\frac{q_k}{2}} 
\geq
F(x^n) =  \|A x^n - b\|_2^2 + 
2 \displaystyle\sum_{k=1}^N \lambda_k |x^n_k|^{q_k} \geq 0 ~.~~~~~~~~
\]
\vspace*{-.3 in}

\end{proof}

\begin{lemma}
\label{lem:xndiffbounded_and_xnbounded}
Assume that the spectral norm of $A$ is bounded by 1, i.e. $\|A\|_2 < 1$. Then the sequence of iterates 
$\{x^n\}$ generated by \eqref{eq:irls_scheme} satisfies 
$\|x^n - x^{n-1}\|_2 \rightarrow 0$ and the $x_n$ are bounded in $\ell_1$-norm 
($\|x^n\|_1 \leq K$ for some $K \in \mathbb{R}$). 
\end{lemma}
\begin{proof}
Using the results from Lemma \ref{lem:irls_surrogate_functional},
we write down a sequence of inequalities:
\begin{eqnarray*}
G(x^{n+1},x^{n+1},w^{n+1},\epsilon_{n+1}) &\leq& G(x^{n+1},x^{n+1},w^{n},\epsilon_{n+1}) \quad [A] \\
&\leq& G(x^{n+1},x^n,w^n,\epsilon_{n+1}) \quad [B] \\
&\leq& G(x^{n+1},x^n,w^n,\epsilon_{n}) \quad [C] \\
&\leq& G(x^n,x^n,w^n,\epsilon_n). \quad [D]
\end{eqnarray*}
We now offer explanations for $[A-D]$. 
First, $[A]$ follows from $w^{n+1} = \displaystyle\arg\min_w G(x^{n+1},a,w,\epsilon_{n+1})$. 
Next for $[B]$, we have:
\begin{equation}
\label{eq:Gxnxn1diff}
G(x^{n+1},x^{n},w^{n},\epsilon_{n+1}) - G(x^{n+1},x^{n+1},w^{n},\epsilon_{n+1}) = 
\|x^n - x^{n+1}\|_2^2  - \|A(x^n - x^{n+1})\|_2^2,
\end{equation}
Now $\|A(x - x^n)\|_2 \leq \|A\|_2 \|x - x^n\|_2 < \|x - x^n\|_2$ 
for $\|A\|_2 < 1$, so that  $\|x - x^n\|_2^2 - \|A (x - x^n)\|_2^2 > 0$. 
Next, $[C]$ follows from $\epsilon_{n+1}\leq\epsilon_{n}$ 
(directly from \eqref{eq:irls_scheme_epsilons}). 
Finally, $[D]$ follows from $x^{n+1} = \displaystyle\arg\min_x G(x,x^n,w^n,\epsilon_n)$. 

\noindent
We now set up a telescoping sum of non-negative terms, using the inequalities 
$[A-D]$ above:
\begin{eqnarray*}
&& \displaystyle\sum_{n=1}^{P} \left( G(x^{n+1},x^{n},w^{n},\epsilon_{n+1}) - G(x^{n+1},x^{n+1},w^{n},\epsilon_{n+1}) \right) 
\\
&& \qquad \leq \displaystyle\sum_{n=1}^{P} \left( G(x^n,x^n,w^n,\epsilon_n) - G(x^{n+1},x^{n+1},w^{n+1},\epsilon_{n+1}) \right) 
\\ 
&& \qquad = G(x^1,x^1,w^1,\epsilon_1) - G(x^{P+1},x^{P+1},w^{P+1},\epsilon_{P+1}) \leq 
G(x^1,x^1,w^1,\epsilon_1) =: C \in \mathbb{R}~,
\end{eqnarray*} 
where we have used that $G(x^n,x^n,w^n,\epsilon_n)$ is always $\geq 0$.
Using \eqref{eq:Gxnxn1diff}, it follows that:
\begin{equation*}
\sum_{n=1}^{P} \left( \|x^n - x^{n+1}\|_2^2  - \|A(x^n - x^{n+1})\|_2^2 \right) \leq C.
\end{equation*}
Since $\|A(x^n - x^{n+1})\|_2^2 \leq \|A\|_2^2 \|x^n - x^{n+1}\|_2^2$ and $\|A\|_2 < 1$:
\begin{eqnarray*}
\|x^n - x^{n+1}\|_2^2  - \|A(x^n - x^{n+1})\|_2^2 
&\geq& \|x^n - x^{n+1}\|_2^2 - \|A\|_2^2 \|x^n - x^{n+1}\|^2_2 \\
&=& \gamma \|x^n - x^{n+1}\|^2_2,
\end{eqnarray*}
where $\gamma := (1 - \|A\|_2^2) > 0$. Consequently, we have:
\begin{equation*}
\begin{array}{cl}
\gamma \displaystyle\sum_{n=1}^{P} \|x^n - x^{n+1}\|_2^2 &\leq\ \displaystyle\sum_{n=1}^{P} \left( \|x^n - x^{n+1}\|_2^2  - \|A(x^n - x^{n+1})\|_2^2 \right) \leq C
\\
  \implies &
\displaystyle\sum_{n=1}^{\infty} \|x^n - x^{n+1}\|_2^2 
< \infty 
\\
  \implies& \|x^n - x^{n+1}\|_2 \rightarrow\ 0.
\end{array}
\end{equation*}

\vspace{3.mm}
To prove that the $\{x^n\}$ are bounded, we use the result from Remark 
\ref{rem:surrogate_and_functional}:
\begin{equation*}
G(x^n,x^n,w^n,\epsilon_n) = \|A x^n - b\|_2^2 + 2 \displaystyle\sum_{k=1}^N \lambda_k \left( (x^n_k)^2 + (\epsilon_n)^2 \right)^\frac{q_k}{2} \geq \lambda_k |x^n_k|^{q_k},
\end{equation*} 
It follows that:
\begin{equation*}
|x^n_k| \leq \left( \frac{1}{\lambda_k} G(x^n,x^n,w^n,\epsilon_n) \right)^\frac{1}{q_k} \leq \left( \frac{1}{\lambda_k} G(x^1,x^1,w^1,\epsilon_1) \right)^\frac{1}{q_k} 
\leq \max_{k \in \{1,\ldots,N\}}\left( \frac{1}{\lambda_k} G(x^1,x^1,w^1,\epsilon_1) \right)^\frac{1}{q_k} =: C_1
\end{equation*}
This implies the boundedness of $\{x^n\}$, since 
$~\|x^n\|_1 = \displaystyle\sum_{k=1}^N |x^n_k| \leq N C_1~$.
\end{proof}
By Lemma \ref{lem:xndiffbounded_and_xnbounded} we have that property (2) holds; moreover (4) (the
boundedness of the $\|x^n\|_1$) is established as well. The next lemma demonstrates property (3) 
and the existence of a convergent subsequence $x^{n_{l_r}}$. 

\vspace{3.mm}
\begin{lemma}
\label{lem:xnl_sequence}
There exists a subsequence $\{\epsilon_{n_l}\}$ of $\{\epsilon_n\}$ such that every 
member of the subsequence is defined by:
\begin{equation*}
\epsilon_{n_l} = \left( \|x^{n_l} - x^{n_l - 1}\|_2 + \alpha^{n_l} \right)^{\frac{1}{2}} < \epsilon_{n_l-1}.
\end{equation*}  
Additionally, there is a subsequence $\{n_{l_r}\}$ of this subsequence such 
that $\{x^{n_{l_r}}\}_r$ is convergent.
\end{lemma}
\begin{proof}
By the definition of the $\epsilon_n$'s in \eqref{eq:irls_scheme_epsilons} and 
by Lemma \ref{lem:xndiffbounded_and_xnbounded},
we know that $\epsilon_n \rightarrow 0$, since $\|x^n-x^{n-1}\|\to0$ and $\alpha^n\to0$. 
It follows that a subsequence $\{n_l\}$ must exist such that 
$\epsilon_{n_l}<\epsilon_{n_l-1}$, for otherwise, the monotonicity 
$\epsilon_{n+1} \leq \epsilon_n$ combined with $\epsilon_n > 0$ for all $n$ would 
imply the existence of $N_0$ such that for $n\geq N_0$, $\epsilon_{n+1} = \epsilon_n$, 
implying that the sequence of $\epsilon_n$'s would not converge to zero. The fact 
that $n_{l_r}$ exists is a consequence of the 
boundedness of the iterates $\{x^n\}$ and hence that of $\{x^{n_l}\}$, 
Lemma \ref{lem:xndiffbounded_and_xnbounded}, and the standard fact that any bounded 
sequence in $\mathbb{R}^N$ has at least one accumulation point.  
\end{proof}
By Lemma \ref{lem:xnl_sequence} and Lemma \ref{lem:xndiffbounded_and_xnbounded}, 
we have that $\epsilon_{n_l} \to 0$ as $l \to \infty$. Thus, together with  
\eqref{eq:gxnxnwnen}, it follows that (3) holds.

\vspace{3.mm}
\begin{lemma}
\label{lem:xnlrconvtominimizer}
The limit $\bar{x}$ of the converging subsequence $\{x^{n_{l_r}}\}$ 
satisfies the optimality conditions \eqref{eq:optcond} of the convex functional 
\eqref{eq:ellq_funct}:
\begin{equation}
\begin{array}{rll}
\{A^T (b - Ax)\}_k &= \lambda_k \sgn(x_k) q_k |x_k|^{q_k-1} \,,& x_k \neq 0 \quad 
(1 \leq q_k \leq 2) \\                         
\left|\{A^T (b - Ax)\}_k\right| &\leq \lambda_k \,,& x_k = 0 \quad (q_k = 1) \\
\{A^T (b - Ax)\}_k &= 0 \,,& x_k = 0 \quad (q_k > 1)                           
\end{array}
\end{equation}
\end{lemma}
\begin{proof}
For each $k$, we consider three separate cases, 
depending on the limit $\bar{x}_k$.
\begin{itemize}
\item[(1)]
	$\overline{x}_k \neq 0$ and $1 \leq q_k \leq 2$,
\item[(2)] 
	$\overline{x}_k = 0$ and $q_k = 1$, 
\item[(3)]  
	$\overline{x}_k = 0$ and $q_k > 1$. 
\end{itemize}
\vspace{3.mm}

Since $x^{n_{l_r}} \rightarrow \bar{x}$, and since $\|x^n - x^{n+1}\| \rightarrow 0$
(by Lemma \ref{lem:xndiffbounded_and_xnbounded}),
we have that:
$x^{n_{l_r} + 1}_k \rightarrow \bar{x}_k$. We can rewrite the iterative 
scheme \eqref{eq:irls_scheme} as:
\begin{equation*}
x^{n+1}_k \left(1 + \lambda_k q_k w^n_k\right) = x^n_k + \{A^T (b - A x^n)\}_k 
\end{equation*}
Specializing this to $\{x^{n_{l_r}}\}$ and reordering terms, we have:
\begin{equation*}
\lambda_k q_k w^{n_{l_r}}_k x^{n_{l_r}+1}_k = x^{n_{l_r}}_k - x^{n_{l_r}+1}_k  + \{A^T ( b - A x^{n_{l_r}})\}_k. 
\end{equation*}
Since the right hand side converges to a limit as $r \to \infty$, so must the left 
hand side; we obtain:
\begin{equation}
\label{eq:limwx}
\lim_{r \to \infty} w^{n_{l_r}}_k x^{n_{l_r}+1}_k = \frac{1}{\lambda_k q_k} \{A^T (b - A \overline{x})\}_k.
\end{equation}
We will use this to compute $\{A^T (b - A\bar{x})\}_k$ and to verify 
that \eqref{eq:optcond} is satisfied.
We are thus interested in the value of 
$\lim_{r \to \infty} w^{n_{l_r}}_k x^{n_{l_r}+1}_k$.

In case (1), $\lim_{r \to \infty} x^{n_{l_r}}_k = \bar{x_k} \neq 0$, we obtain
\begin{eqnarray*}
\lim_{r \to \infty} w^{n_{l_r}}_k x^{n_{l_r}+1}_k = \lim_{l \to \infty} w^{n_{l_r}}_k x^{n_{l_r}}_k  \frac{x^{n_{l_r}+1}_k}{x^{n_{l_r}}_k} = \lim_{l \to \infty} x^{n_{l_r}}_k w^{n_{l_r}}_k,
\end{eqnarray*}
where we have used that $\lim_{r \to \infty} \frac{x^{n_{l_r}+1}_k}{x^{n_{l_r}}_k} = 1$, 
since $\|x^{n+1} - x^n\| \rightarrow 0$. Using \eqref{eq:irls_scheme_weights}, it 
follows that:
\begin{eqnarray*}
\lim_{r \to \infty} w^{n_{l_r}}_k x^{n_{l_r}+1}_k 
&=& \lim_{r \to \infty} \frac{x^{n_{l_r}}_k}{\left[(x^{n_{l_r}}_k)^2 + (\epsilon_{n_{l_r}})^2\right]^{\frac{2-q_k}{2}}} 
= \frac{\overline{x_k}}{\left((\overline{x_k})^2 + 0\right)^{\frac{2-q_k}{2}}} \\ 
&=& \frac{\sgn(\overline{x_k})\left|\overline{x_k}\right| }{\left|\overline{x_k}\right|^{2-q_k}} 
= \sgn(\overline{x_k}) |\overline{x_k}|^{q_k-1}.
\end{eqnarray*}
Thus, from \eqref{eq:limwx}, we obtain that: 
$\{A^T (b - A \overline{x})\}_k = \lambda_k q_k \sgn(\overline{x_k}) |x_k|^{q_k -1}$, in 
accordance with \eqref{eq:optcond}.

\vspace{3.mm}
In case (2) and (3), $\lim_{r \to \infty} x^{n_{l_r}} = \overline{x_k} = 0$, and we 
still have that \eqref{eq:limwx} holds. Writing out \eqref{eq:irls_scheme} for 
$x^{n_{k_{l_r}}}$ in terms of $x^{n_{k_{l_r}}-1}$, we obtain:
\begin{equation*}
\lambda_k q_k w^{n_{l_r}-1}_k x^{n_{l_r}}_k = x^{n_{l_r}-1}_k - x^{n_{l_r}}_k  + \{A^T ( b - A x^{n_{l_r}-1})\}_k. 
\end{equation*}
which gives the limit:
\begin{equation}
\label{eq:limwx2}
\lim_{r \to \infty} w^{n_{l_r}-1}_k x^{n_{l_r}}_k = \frac{1}{\lambda_k q_k} \{A^T (b - A \overline{x})\}_k.
\end{equation}
We define $\beta_k$ to be: 
\begin{equation*}
\beta_k := \frac{1}{\lambda_k q_k} \{A^T (b - A \overline{x})\}_k 
\end{equation*}
To prove that \eqref{eq:optcond} is satisfied, we must show that 
$|\beta_k| \leq 1$ for case (2) and that $\beta_k = 0$ for case (3). 

We first write down some relations involving $\beta_k$ which we will use. Note that 
by \eqref{eq:limwx2},  
$\lim_{r \to \infty} w^{n_{l_r}-1}_k x^{n_{l_r}}_k = \beta_k$. 
If $\beta_k \neq 0$, it follows that  for every $\sigma \in (0,1), \exists r_0$ such that for every 
$r \geq r_0$: 
\begin{equation*}
\left( w^{n_{l_r}-1}_k x^{n_{l_r}}_k \right)^2 > (1 - \sigma) \beta_k^2 \implies \left( x^{n_{l_r}}_k \right)^2 > (1 - \sigma) \beta_k^2 \left(w^{n_{l_r}-1}_k\right)^{-2} =  (1 - \sigma) \beta_k^2 \left( (x^{n_{l_r}-1}_k)^2 + (\epsilon_{n_{l_r}-1})^2 \right)^{2-q_k} 
\end{equation*}
Since $\epsilon_{n_{l_r}} < \epsilon_{n_{l_r}-1}$, it follows that for $r$ sufficiently 
large::
\begin{eqnarray*}
\left( x^{n_{l_r}}_k \right)^2 &>& (1-\sigma) \beta_k^2 \left( (x^{n_{l_r}-1}_k)^2 + (\epsilon_{n_{l_r}})^2 \right)^{2-q_k} \\
&=& (1-\sigma) \beta_k^2 \left( (x^{n_{l_r}-1}_k)^2 + \|x^{n_{l_r}} - x^{n_{l_r}-1}\|_2 + \alpha^{n_{l_r}} \right)^{2-q_k} \\
&\geq& (1-\sigma) \beta_k^2 \left( (x^{n_{l_r}-1}_k)^2 + |x^{n_{l_r}}_k - x^{n_{l_r}-1}_k| + \alpha^{n_{l_r}} \right)^{2-q_k} \\
&>&  (1-\sigma) \beta_k^2 \left( (x^{n_{l_r}-1}_k)^2 + |x^{n_{l_r}}_k - x^{n_{l_r}-1}_k| \right)^{2-q_k}  
\end{eqnarray*}
where we have used in the last part that $\alpha^{n_{l_r}} \to 0$. 
To simplify notation, let us set $u = x^{n_{l_r}-1}_k$ and 
$v = x^{n_{l_r}}_k - x^{n_{l_r}-1}_k$. Then in terms of $u$ and $v$, we have:
\begin{equation}
\label{eq:u_plus_v_relation}
(u + v)^2 > (1 - \sigma) \beta_k^2 \left( u^2 + |v| \right)^{2-q_k} 
\end{equation}
Notice that for any $K > 0$: 
\begin{equation*}
0 \leq \left( \sqrt{K} u - \frac{1}{\sqrt{K}} v \right)^2 = K u^2 + \frac{1}{K} v^2 - 2 u v 
\end{equation*}
It follows that:
\begin{equation}
\label{eq:u_plus_v_squared_K_relation}
(u + v)^2 = u^2 + 2 u v + v^2 \leq u^2 + K u^2 + \frac{1}{K} v^2 + v^2 = (1 + K) u^2 + \left(1 + \frac{1}{K}\right) v^2
\end{equation}
Using \eqref{eq:u_plus_v_squared_K_relation} in \eqref{eq:u_plus_v_relation}, we 
get:
\begin{equation}
\label{eq:u_and_v_main_relation}
(1 - \sigma) \beta_k^2 \left( u^2 + |v| \right)^{2-q_k} < (1 + K) u^2 + (1 + \frac{1}{K}) v^2  
\end{equation}

\vspace{3.mm}
Let us now consider case (2) where $q_k = 1$. We assume that $\beta_k > 1$ and 
derive a contradiction. Rearranging terms in \eqref{eq:u_and_v_main_relation} yields:
\begin{equation}
\label{eq:u_and_v_final_relation_case2}
\left( (1 - \sigma) \beta_k^2 - (1 + K) \right) u^2 < (1 + \frac{1}{K}) v^2 - (1 - \sigma) \beta_k^2 |v| = \left( (1 + \frac{1}{K}) |v| - (1 - \sigma) \beta_k^2  \right) |v| 
\end{equation}
Since we assume that $|\beta_k^2| > 1$, we can choose our $\sigma < 1$ 
small enough such that $(1 - \sigma) \beta_k^2 > 1$; once $\sigma$ is fixed, we can  
choose $K > 0$ small enough such that $(1 - \sigma) \beta_k^2 \geq (1 + K)$; 
with these choices of $\sigma$ and $K$, the left hand side 
of \eqref{eq:u_and_v_final_relation_case2} $\geq 0$. 
With this fixed choice of $\sigma$ and $K$ we analyze the right hand side of 
\eqref{eq:u_and_v_final_relation_case2}. Note that by Lemma 
\ref{lem:xndiffbounded_and_xnbounded}, we have that 
$\|x^{n_{l_r}} - x^{n_{l_r}-1}\|_2 \to 0$ as $r \to \infty$. This means that 
$|v| \to 0$ as $r \to \infty$. For sufficiently large $r$, we will have 
$|v| < \left(1 + \frac{1}{K}\right)^{-1} (1 - \sigma) \beta_k^2$, implying 
that the right hand side of \eqref{eq:u_and_v_final_relation_case2} would then be $\leq 0$.
This is in contradiction with the left hand side of this strict inequality 
\eqref{eq:u_and_v_final_relation_case2} being $\geq 0$. 
It follows that the assumption $|\beta_k| > 1$ is not correct. Hence, we have 
$|\beta_k| \leq 1$ which implies that $|\{A^T (b - A \overline{x})\}_k| \leq \lambda_k$, 
consistent with \eqref{eq:optcond}.

\vspace{3.mm}
Finally, consider case (3) with $q_k > 1$. We assume that $|\beta_k| > 0$ and 
derive a contradiction. In this case, 
\eqref{eq:u_and_v_main_relation} does not simplify further:
\begin{equation}
\label{eq:u_and_v_final_relation_case3}
\beta_k^2 (1 - \sigma) \left( u^2 + |v| \right)^{2-q_k} < (1 + K) u^2 + \left(1 + \frac{1}{K}\right) v^2 \quad \mbox{for all } K > 0 .
\end{equation}
This means in particular that:
\begin{eqnarray*}
\beta_k^2 (1 - \sigma) u^{2(2 - q_k)} &<& (1 + K) u^2 + \left(1 + \frac{1}{K}\right) v^2 \quad \mbox{and} \\ 
\beta_k^2 (1 - \sigma) |v|^{(2 - q_k)} &<& (1 + K) u^2 + \left(1 + \frac{1}{K}\right) v^2.
\end{eqnarray*}
Then the average of the terms is also smaller than this quantity:
\begin{equation*}
\frac{1}{2} \beta_k^2 (1 - \sigma) \left( u^{2(2 - q_k)} + |v|^{(2 - q_k)} \right) 
< (1 + K) u^2 + \left(1 + \frac{1}{K}\right) v^2.
\end{equation*}
Rearranging terms again, we have:
\begin{equation*}
u^{2(2 - q_k)} \left( \frac{1}{2} \beta_k^2 (1 - \sigma)  - (1 + K) u^{2 (q_k - 1)} \right) 
< \left(1 + \frac{1}{K}\right) v^2 - \frac{1}{2} \beta_k^2 (1 - \sigma) |v|^{(2 - q_k)}.
\end{equation*}
Since $q_k > 1$ and thus $2 - q_k < 1$, we have that for $v$ sufficiently small 
(obtained by taking $r$ sufficiently large), the right hand side is negative, 
by the same logic as in the previous case 
(because $|\beta_k| > 0$ by assumption, the first term will go to zero faster than the second
when $v \rightarrow 0$ as $r \rightarrow \infty$).
Thus, by the above inequality, for $r$ sufficiently large, the left hand side, 
bounded above by the negative right hand side, must be negative as well. 
Since $u^{2(2 - q_k)}$ is non-negative, that is possible only when:
\begin{equation*}
\frac{1}{2} \beta_k^2 (1 - \sigma)  - (1 + K) u^{2 (q_k - 1)} < 0
\end{equation*} 
for $r$ sufficiently large. However, since $\lim_{r \to \infty} u^{2 (q_k - 1)} = \lim_{r \to \infty} (x^{n_{l_r}-1}_k)^{2 (q_k - 1)}    = 0$, this condition cannot be 
satisfied for large $r$. This contradicts our original assumption that $|\beta_k| > 0$. 
Hence, we conclude that $\beta_k = 0$. It follows that $\{A^T(b - A \overline{x})\}_k = 0$, 
which is the right optimality condition. 
\end{proof}

Lemma \ref{lem:xnlrconvtominimizer}, together with the proceeding Lemmas in this 
section, show that properties (1) to (5) of Section \ref{subsect:argflow} hold. 
It thus follows from the argument in Section \ref{subsect:argflow} that 
we have $F(x^n) \to F(\bar{x})$.

\vspace{3.mm}
\section{Numerics}

We now discuss some aspects of the numerical implementation and performance 
of the IRLS algorithm. We first illustrate performance for the case 
$q_k = 1$ for all $k$, where it's easiest to compare with existing algorithms. 
Then we discuss a simple example concerning a case where different values of 
$q_k$ can be used. An implementation of the scheme as given 
by \eqref{eq:irls_scheme} has the same computational complexity as ISTA  
in \eqref{eq:ista_scheme}. Not surprisingly, the performance of the 
two schemes is also similar. However, our numerical experiments indicate 
that the speed-up idea behind FISTA 
as described in \cite{Beck2009} is also effective for the IRLS 
algorithm. FISTA was designed to minimize the function $f(x)+g(x)$,
where $f$ is a continuously differentiable convex function with Lipschitz continuous 
gradient (i.e., $\|\nabla f(x)-\nabla f(y)\|_2 \leq L \|x-y\|_2$ for some constant $L > 0$),
and $g$ is a continuous convex function such as $2\tau \|x\|_1$ in the 
$\ell_1$-penalized functional. FISTA uses the proximal mapping function:
\begin{equation*}
p_L(y) = \arg\min_x \left\{ g(x) + \frac{L}{2} \left\|x - (y - \frac{1}{L} \nabla f(y) )\right\|_2^2 \right\}
\end{equation*}
to define the following algorithm:
\begin{equation}
\label{eq:FISTA}
\begin{array}{rcl}
 y^1 &=& x^0 \in \mathbb{R}^N \quad \mbox{,} \quad t_1 = 1 \quad \mbox{,} \quad \mbox{and for  } n=1,2,\dots, \\
\displaystyle x^{n+1} &=& 
\displaystyle p_L(y^n) = 
\arg\min_x \left\{ g(x) + 
\frac{L}{2} \left\|x - (y^n - \frac{1}{L} \nabla f(y^n) )\right\|^2 \right\} \\
 \displaystyle t_{n+1} &=& \displaystyle \frac{1 + \sqrt{1 + 4 t_n^2}}{2} \\
 \displaystyle y^{n+1} &=& \displaystyle x^{n+1} + \frac{t_n-1}{t_{n+1}} (x^{n+1} - x^{n}).
\end{array}
\end{equation}
In the case that $f(x) = \|Ax - b\|_2^2$ and $g(x) = 2 \tau \|x\|_1$, we obtain:
\begin{eqnarray*}
\|\nabla f(x) - \nabla f(y)\|_2 &=& \|2 A^T A x - 2 A^T A y\|_2 = \|2 A^T A (x - y)\|_2 \\ 
&\leq& 2 \|A^T A\|_2 \|x - y\|_2,
\end{eqnarray*} 
which implies that when $A$ is scaled such that $\|A\|_2 \approx 1$, the Lipschitz 
constant can be taken to be $L = 2$. It follows that:
\begin{equation}
\label{eq:FISTAg}
g(x) + \frac{L}{2} \left\|x - (y - \frac{1}{L} \nabla f(y) )\right\|_2^2 = 2 \tau \|x\|_1 + \left\|x - (y - A^T (Ay - b) )\right\|_2^2 .
\end{equation}
Using \eqref{eq:FISTAg} in \eqref{eq:FISTA}, we obtain:
\begin{equation}
\label{eq:FISTAxnp1}
x^{n+1} = \arg\min_x \left\{  2 \tau \|x\|_1 + \left\|x - (y^n - A^T (Ay^n - b) )\right\|_2^2  \right\} = \mathbb{S}_{\tau} \left( y^n - A^T Ay^n + A^T b \right),
\end{equation}
where we have used \eqref{eq:soft_thresholding_min_def}. We note that 
\eqref{eq:FISTAxnp1} is very similar to the ISTA scheme in \eqref{eq:ista_scheme}, 
except the thresholding is applied to $\{y^n\}$. In the same way, we can 
coin the FIRLS algorithm by performing the steps in \eqref{eq:FISTA}, using 
\begin{equation}
\label{eq:FIRLSxnp1}
x^{n+1} = \frac{1}{1 + \tau \left[ (y^n_k)^2  + (\epsilon_n)^2 \right]^{\frac{1}{2}}} \left\{ y^n - A^T Ay^n + A^T b \right\}_k \quad \mbox{for} \quad k=1,\dots,N
\end{equation}
in place of \eqref{eq:FISTAxnp1}. With the more general weights given by  
\eqref{eq:irls_scheme_weights}, we can specialize this algorithm to our 
functional \eqref{eq:ellq_funct}.

We now demonstrate some results of simple numerical experiments. 
We begin with the $q_k = 1$ case for all $k$. We also let the regularization 
parameter be the same for all $k$, setting $\lambda_k = \tau$.
For the first test, we use two differently conditioned random matrices 
(built up via a reverse
SVD procedure with orthogonal random matrices $U$ and $V$, obtained 
by performing a QR factorization on the Gaussian random matrices, and a custom 
diagonal matrix of singular values $S$, to form a $1000 \times 1000$ matrix $A = U S V^T$), 
and a sparse signal $x$ with $5\%$ non-zeros. We form $b = A x$ and 
use the different algorithms to recover $\tilde{x}$ using a single 
run of 300 iterations with $\tau = \frac{\max{|A^T b|}}{10^5}$. 
In Figure \ref{fig:numerics1}, we plot the decrease of 
$\ell_1$-functional values $F_1(x^n)$ and recovery percent errors 
$100 \frac{\|x^n - x\|}{\|x\|}$ versus the iterate number $n$, using 
four algorithms: IRLS, FIRLS, ISTA, and FISTA for two matrix types: $A_1$, with 
singular values logspaced between $1$ and $0.1$ and $A_2$, 
with singular values logspaced between $1$ and $10^{-4}$. 
We see that the performance of ISTA/IRLS and FISTA/FIRLS are mostly similar, with better 
recovery using FISTA in the well-conditioned case, but almost identical performance in 
the worst-conditioned case. 

In Figure \ref{fig:numerics2}, we run a compressive sensing experiment. 
We again take the $1000\times 1000$ matrix of type $A_2$. Now we use a staircase-like 
sparse vector $x$ with about $12 \%$ non-zeros.
After we form $b = A x$, we zero out all but the first $\frac{1}{3}$ of 
the rows of $A$ and $b$ forming $A_p$ and $b_p$ (i.e. we only keep a portion of the 
measurements). We then recover solution 
$\tilde x$ using $A_p$ and $b_p$ while employing a continuation scheme across 
$20$ different values 
of $\tau$, starting with a zero initial guess at $\tau = \max{|A^T b|}$ 
and proceeding down to $\tau = \frac{\max{|A^T b|}}{50000}$, while reusing 
the previous solutions as the initial guess at each new value of $\tau$. 
From Figure \ref{fig:numerics2}, we can see that the recovered solutions with 
FIRLS and FISTA are very similar.

We illustrate the use of the more general functional 
in \eqref{eq:ellq_funct} in Figure \ref{fig:numerics3}. We use the 
same setup as before, with the different algorithms running across 
multiple values of the regularization parameter $\tau$, which is fixed 
for all $k$. However, we use a more complicated input signal, whose first 
half is sparse and whose second half is entirely dense. For this reason, 
in the IRLS schemes, we take $q_k = 1$ for the first half of the weights 
(for indices $k$ from $1$ to $\frac{n}{2}$) and $q_k = 1.9$ for the second half 
(for indices $k$ from $\frac{n}{2}+1$ to $n$). We observe that the recovered signal 
with the IRLS algorithms is superior to that of the ISTA/FISTA schemes 
which utilize $q_k = 1$ for all entries. Of course, setting the values 
of $q_k$ for individual coefficients maybe difficult in practice unless one knows 
the distribution of the sparser and denser parts in advance, although in applications, 
some information of this nature may be available from the setup of the problem.  

Finally, in Figure \ref{fig:numerics4}, we show the result of an image 
reconstruction experiment. We use two images, blurred with a Gaussian source and 
corrupted by Gaussian noise. The first image is $170\times120$ and the second 
is $125\times125$. In both cases, the blurring source is a 2D Gaussian function 
with support on a $9\times9$ grid with $\sigma = 2.5$ and max amplitude of $2.9$. 
We then take the blurred image (obtained via convolution with the blurring source) 
and add white Gaussian 
noise, so that the signal to noise ratio is $25$. We then recover a corrected image 
using an application of wavelet denoising followed by IRLS, from the blurred and 
noisy image. The IRLS algorithms is run over $30$ parameters $\tau$ with 
$40$ iterations each, in a setup similar to that used for Figure \ref{fig:numerics2}. 
The matrices we use in the inversion are derived from the blur source itself, 
so this is a non-blind deconvolution. The problem, however, is still challenging and the 
resulting images are much improved from their blurred and noisy counterparts. 
We have noticed that the use of $q_k < 1$ can yield, in some instances, slighter 
sharper reconstructions in the same number of iterations.

\newpage
\begin{figure}[!ht]
\centerline{
\includegraphics[scale=0.2]{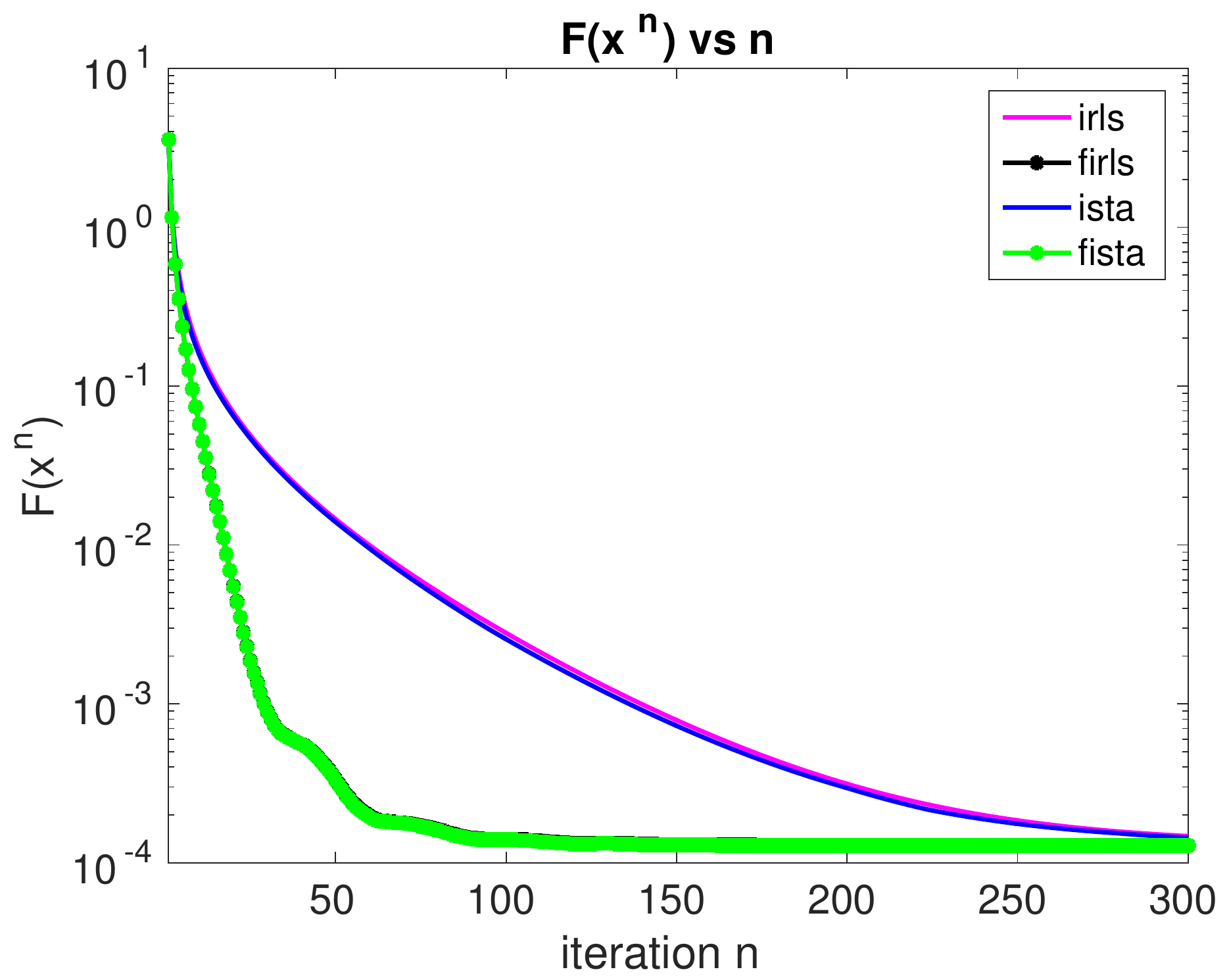}
\includegraphics[scale=0.2]{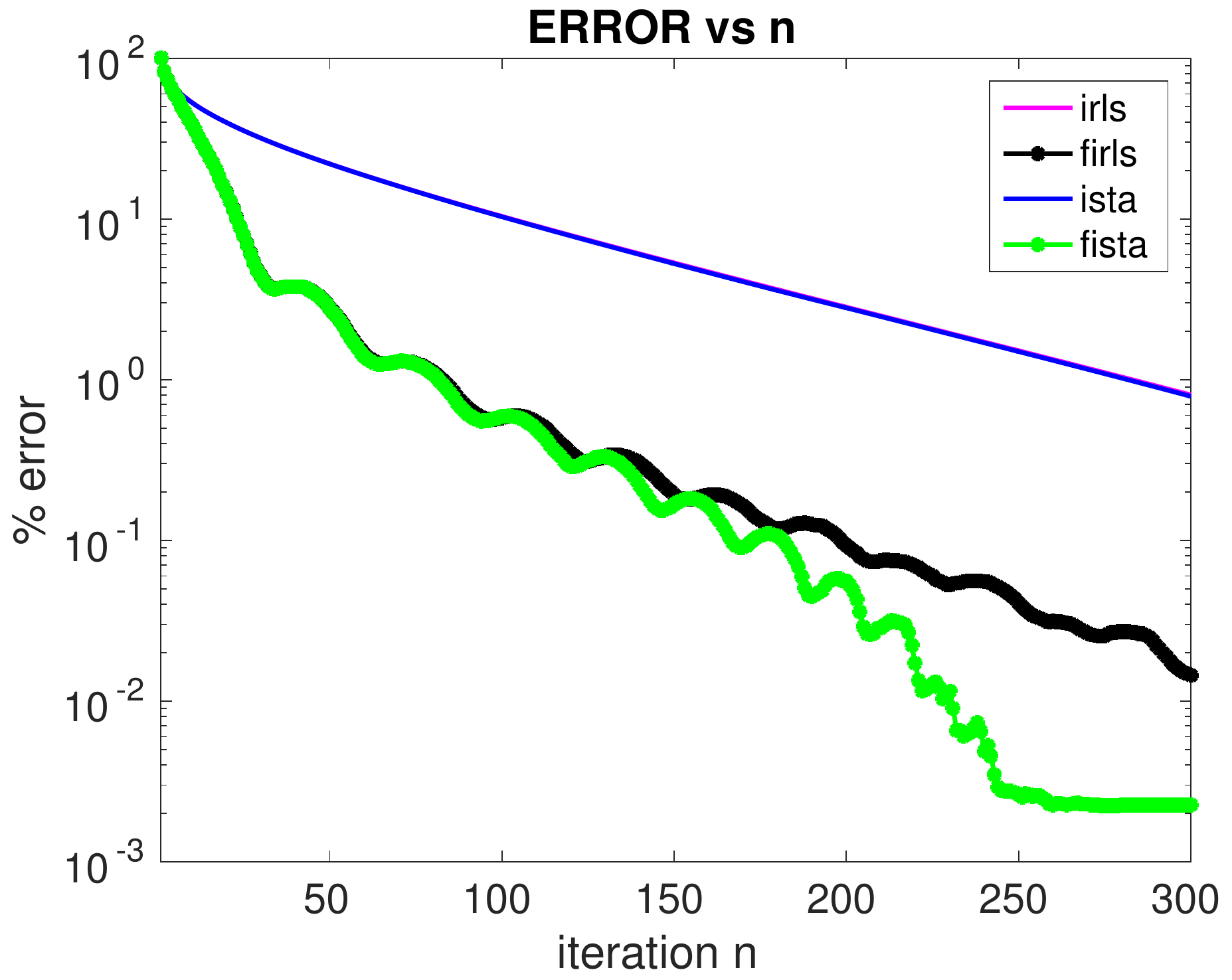}
\quad
\includegraphics[scale=0.2]{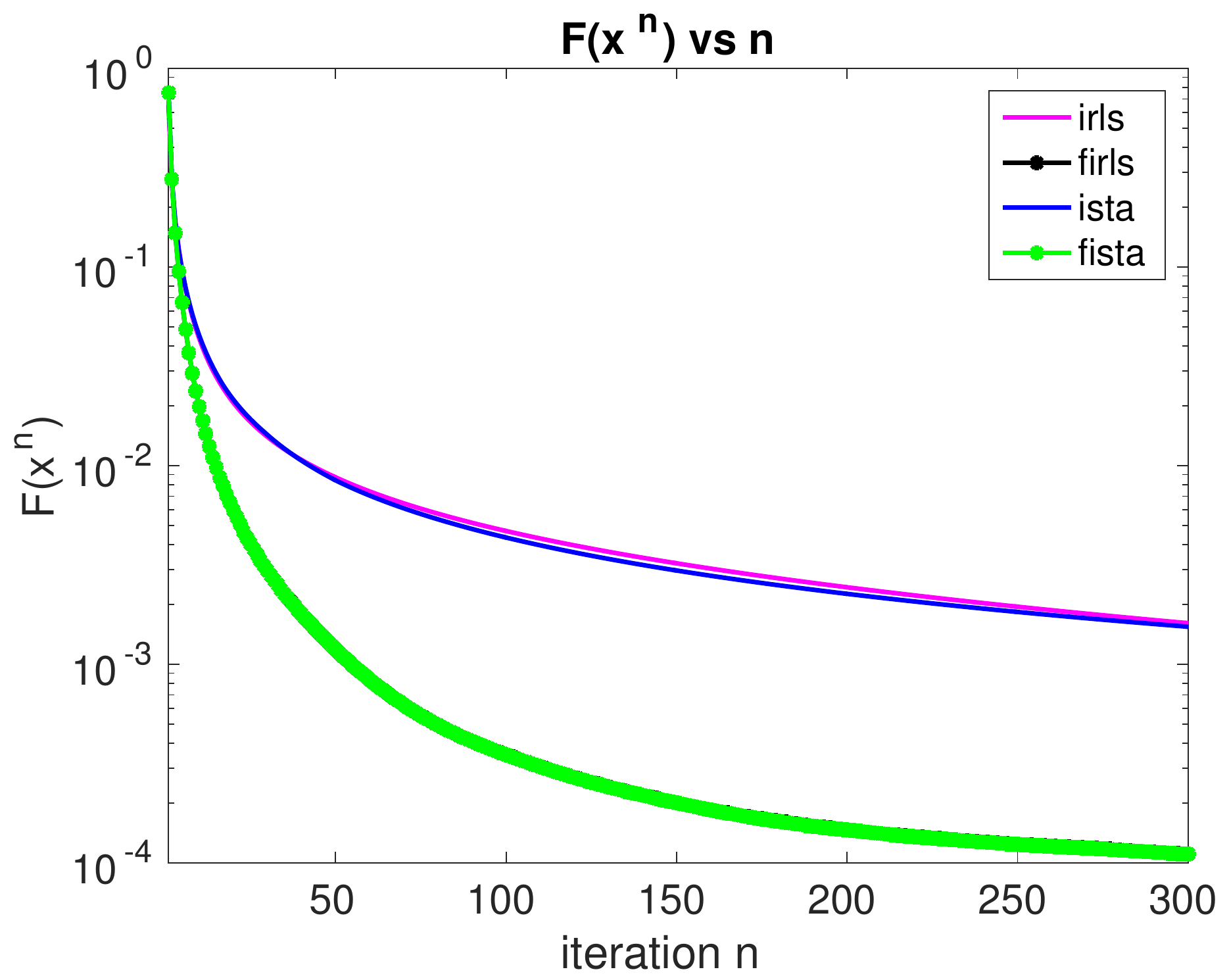}
\includegraphics[scale=0.2]{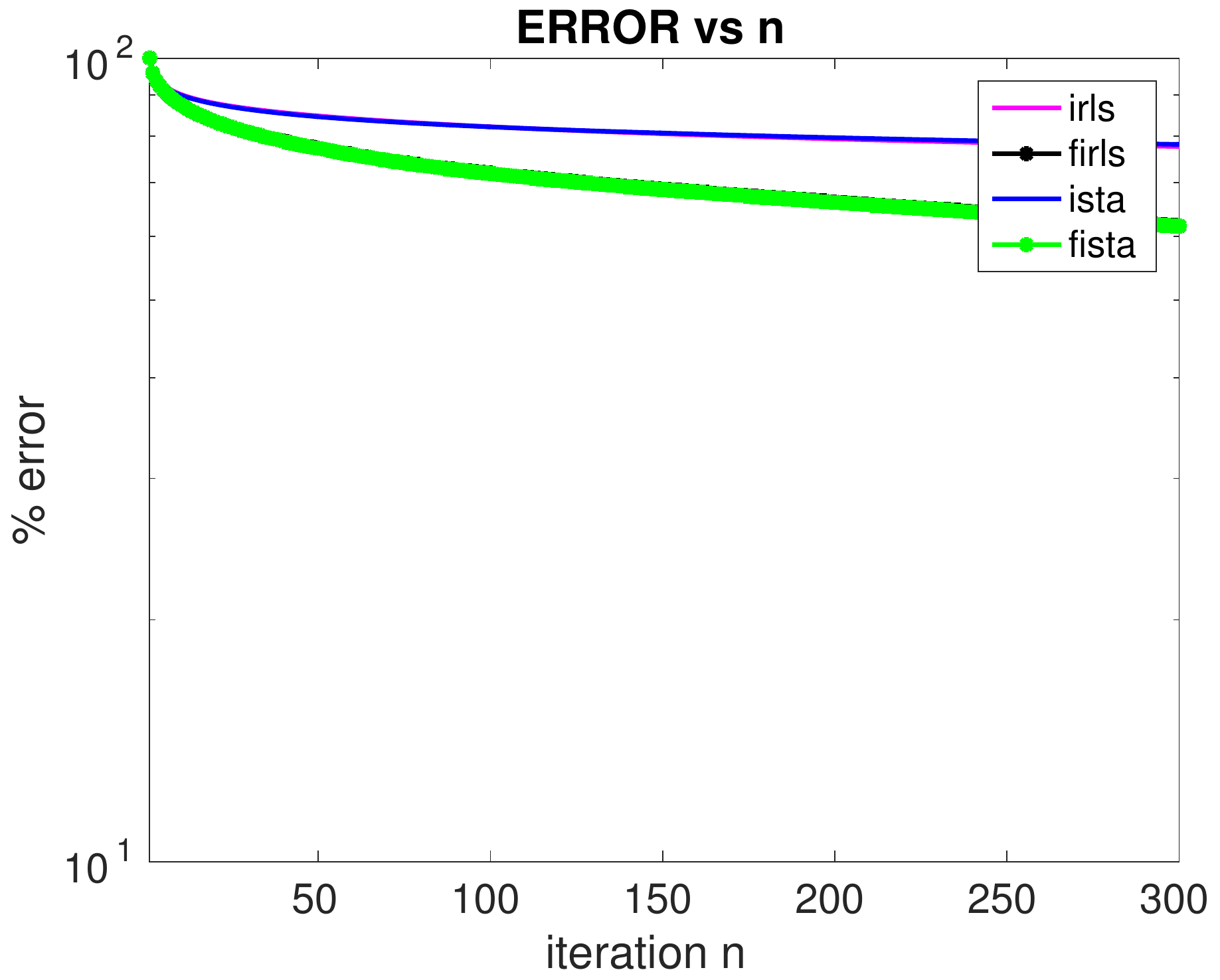}
}
\caption{Functional values $F(x^n)$ and recovery percent errors  
$100 \frac{\|x^n - x\|}{\|x\|}$ versus the iterate number $n$ for better and 
worse conditioned matrices (medians over 10 trials).}
\label{fig:numerics1}
\end{figure}

\vspace{10.mm}

\begin{figure*}[ht!]
\centerline{
\includegraphics[scale=0.16]{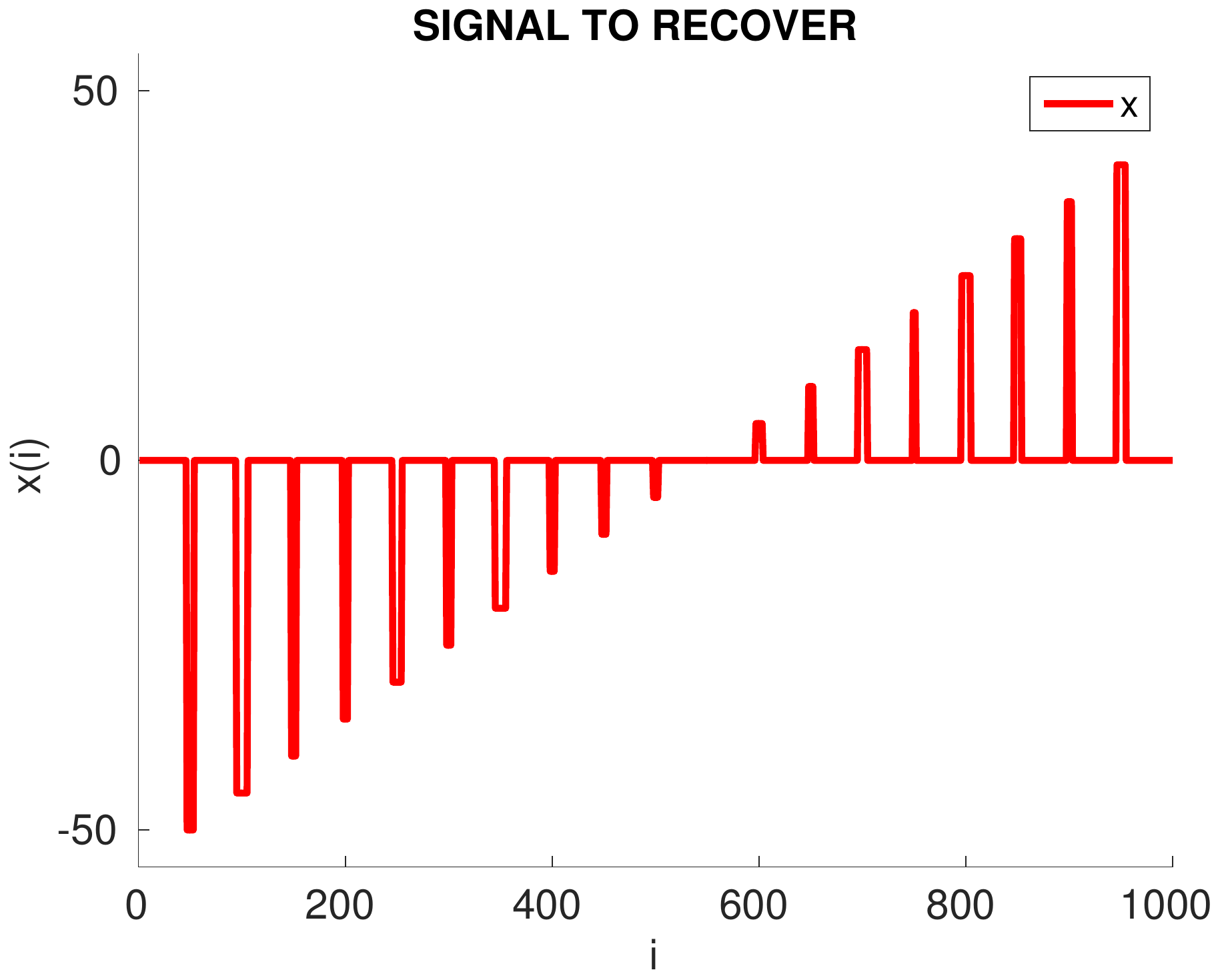}
\includegraphics[scale=0.16]{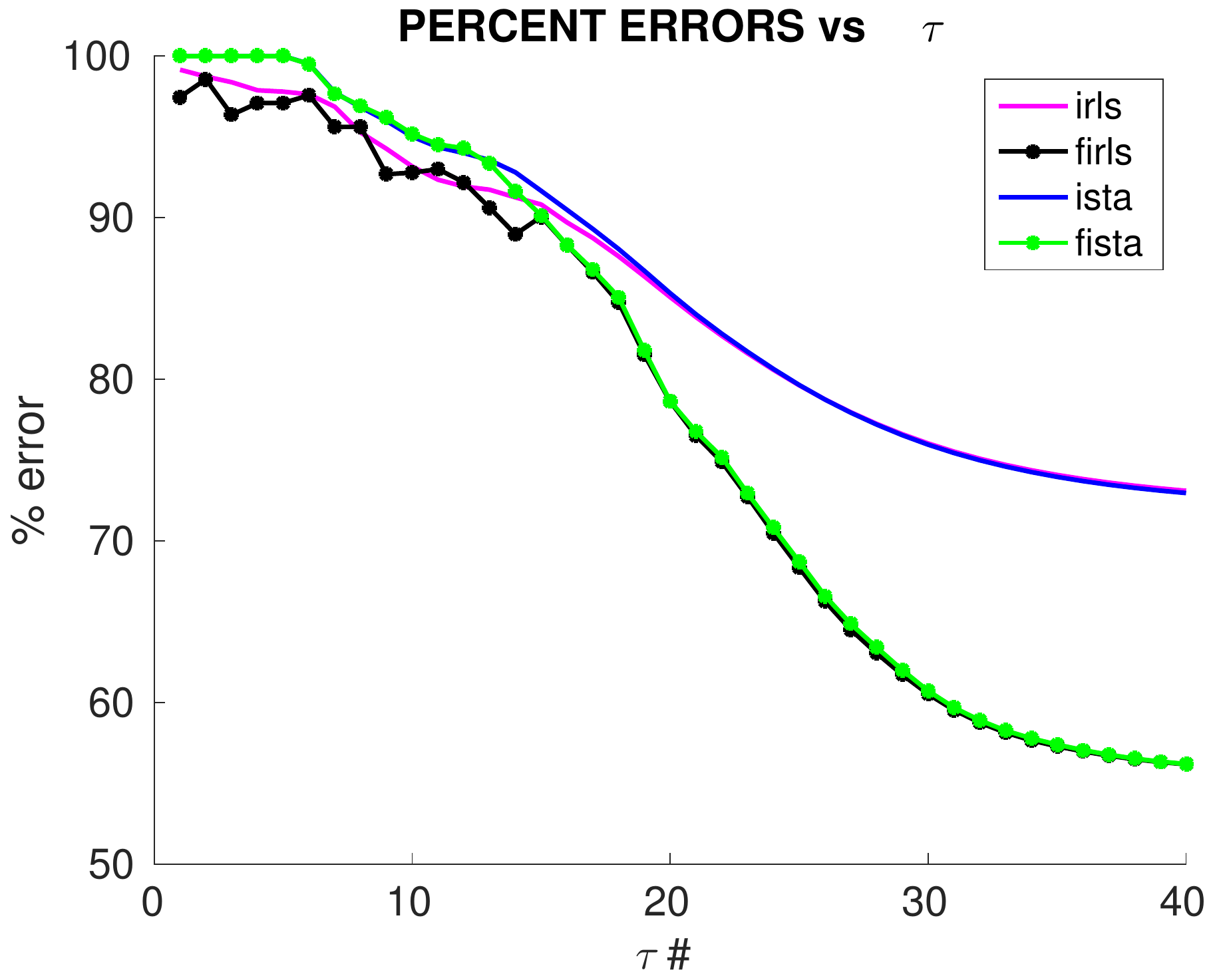}
}
\centerline{
\includegraphics[scale=0.16]{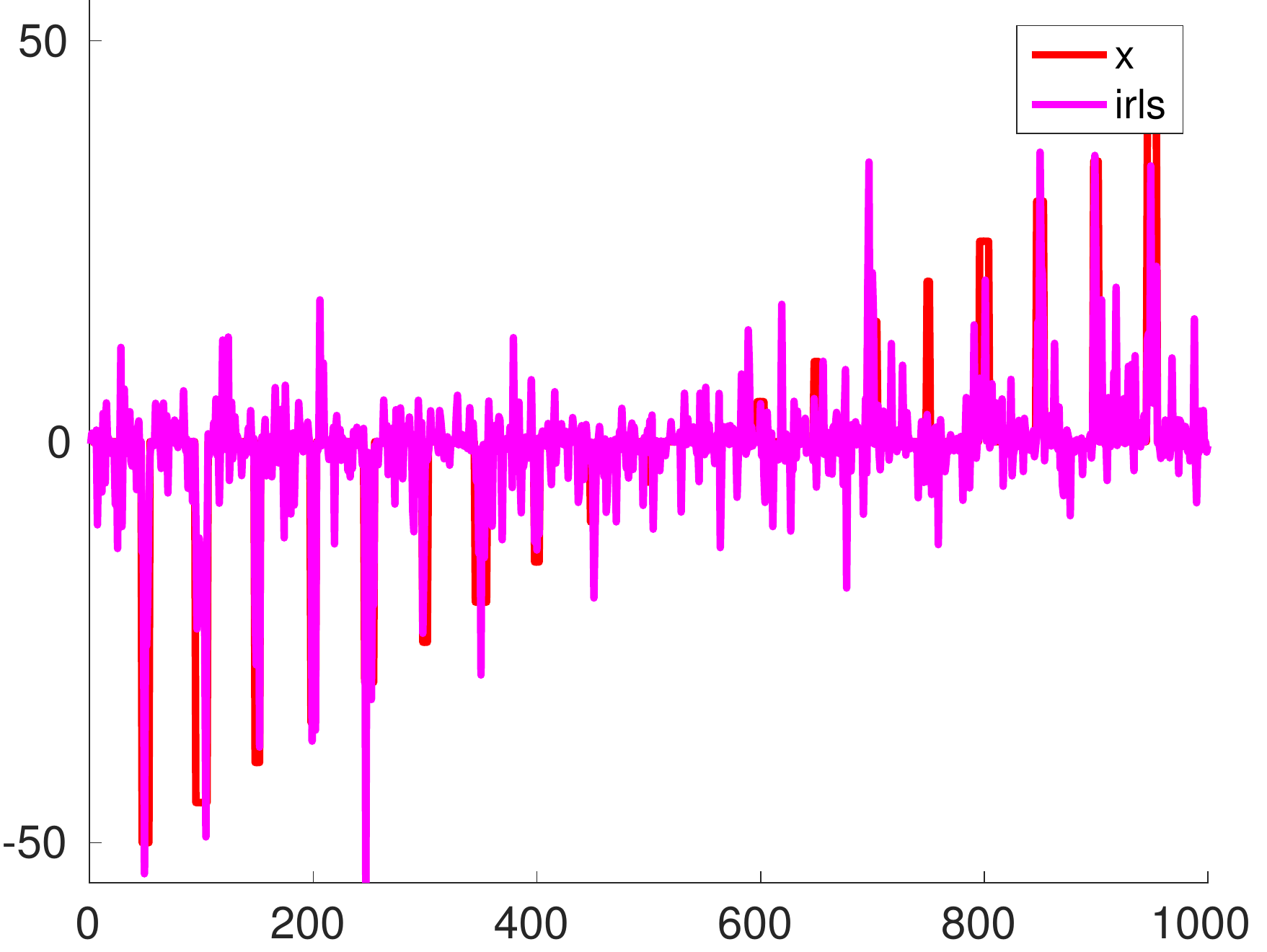}
\includegraphics[scale=0.16]{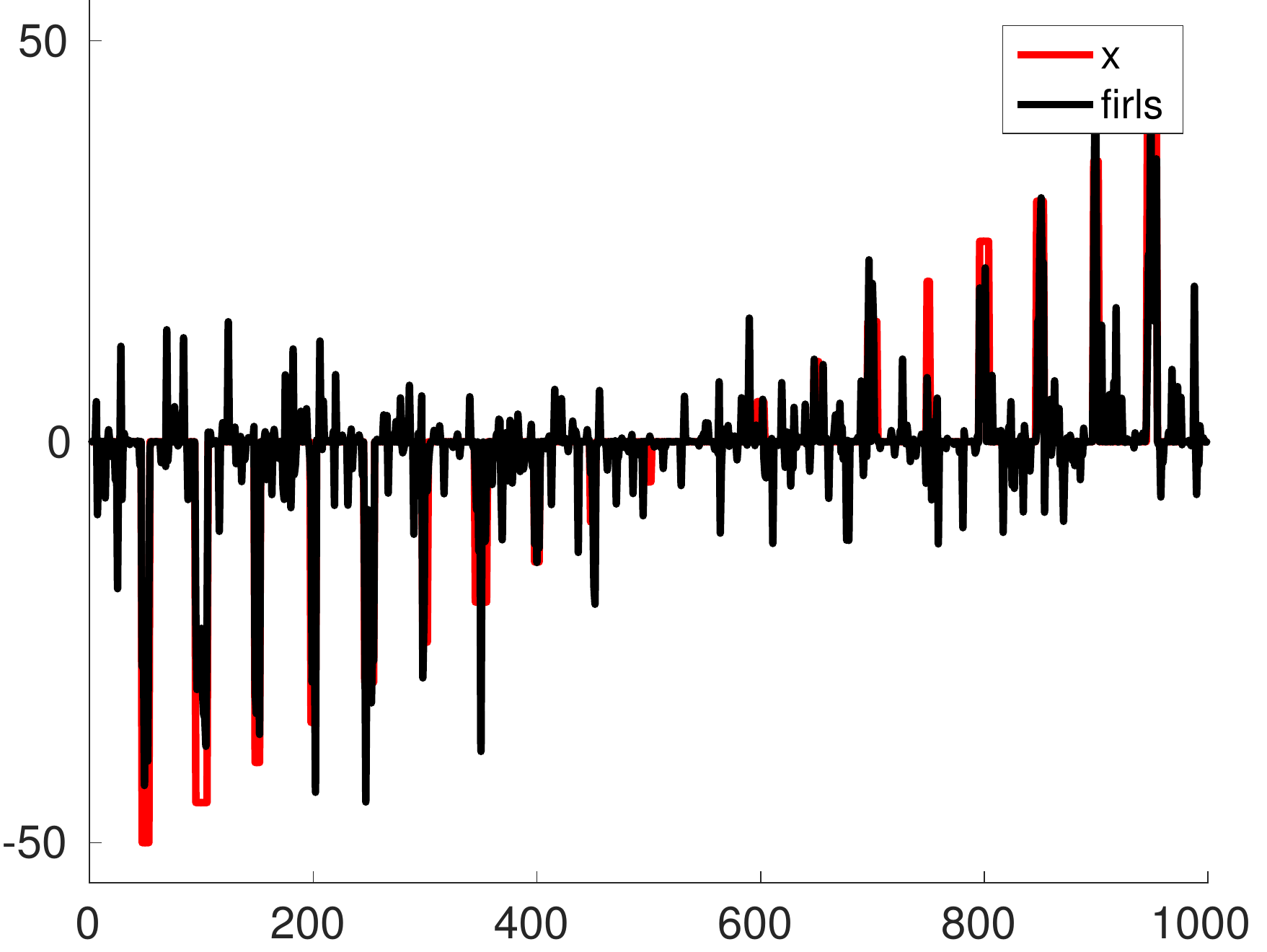}
\includegraphics[scale=0.16]{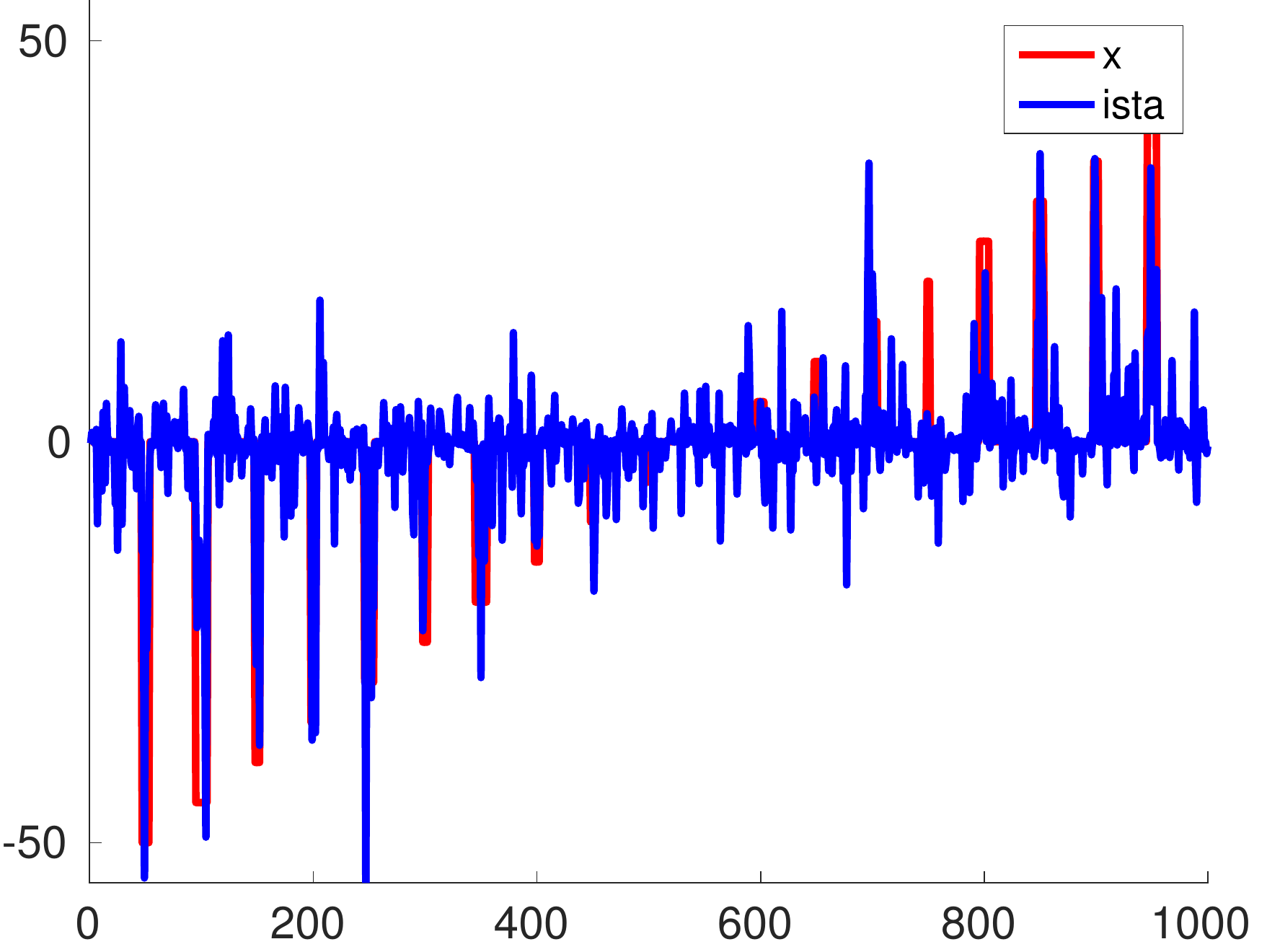}
\includegraphics[scale=0.16]{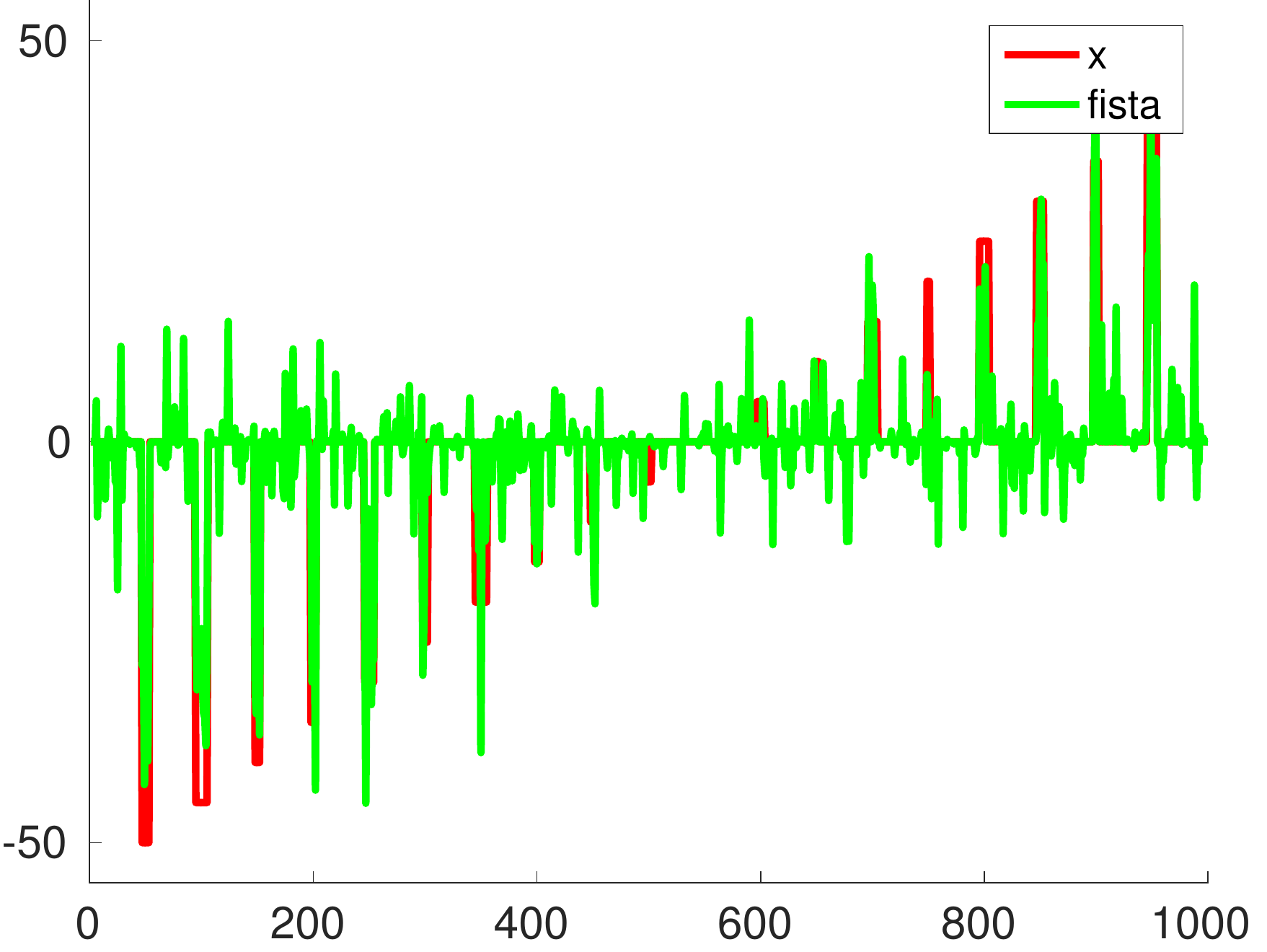}
}
\caption{Row 1: sparse model $x$ and the recovery percent errors vs $\tau$. 
Row 2: final recovered solution with algorithms IRLS, FIRLS, ISTA, FISTA vs $x$.}
\label{fig:numerics2}
\end{figure*}

\vspace{10.mm}

\begin{figure*}[ht!]
\centerline{
\includegraphics[scale=0.16]{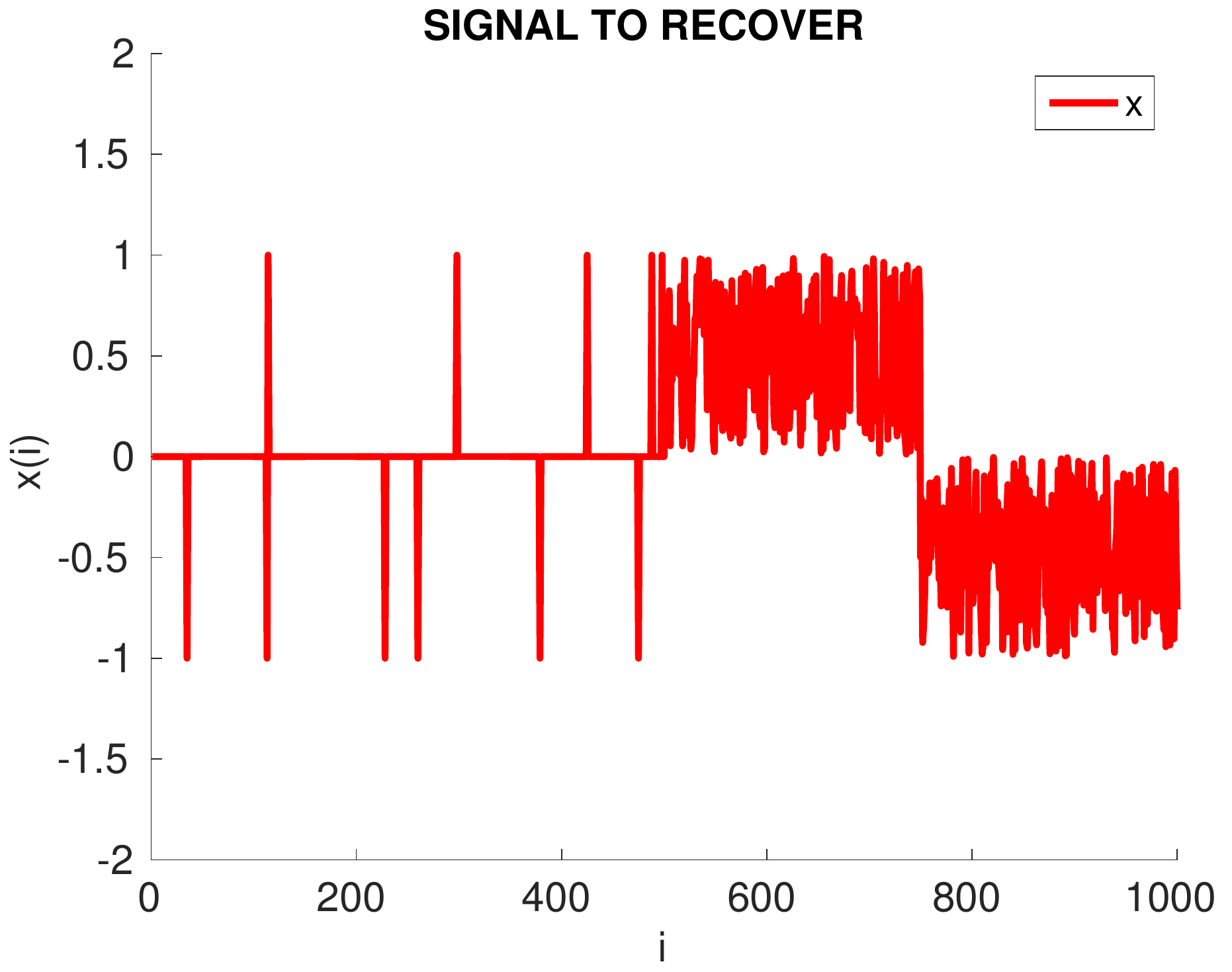}
\includegraphics[scale=0.16]{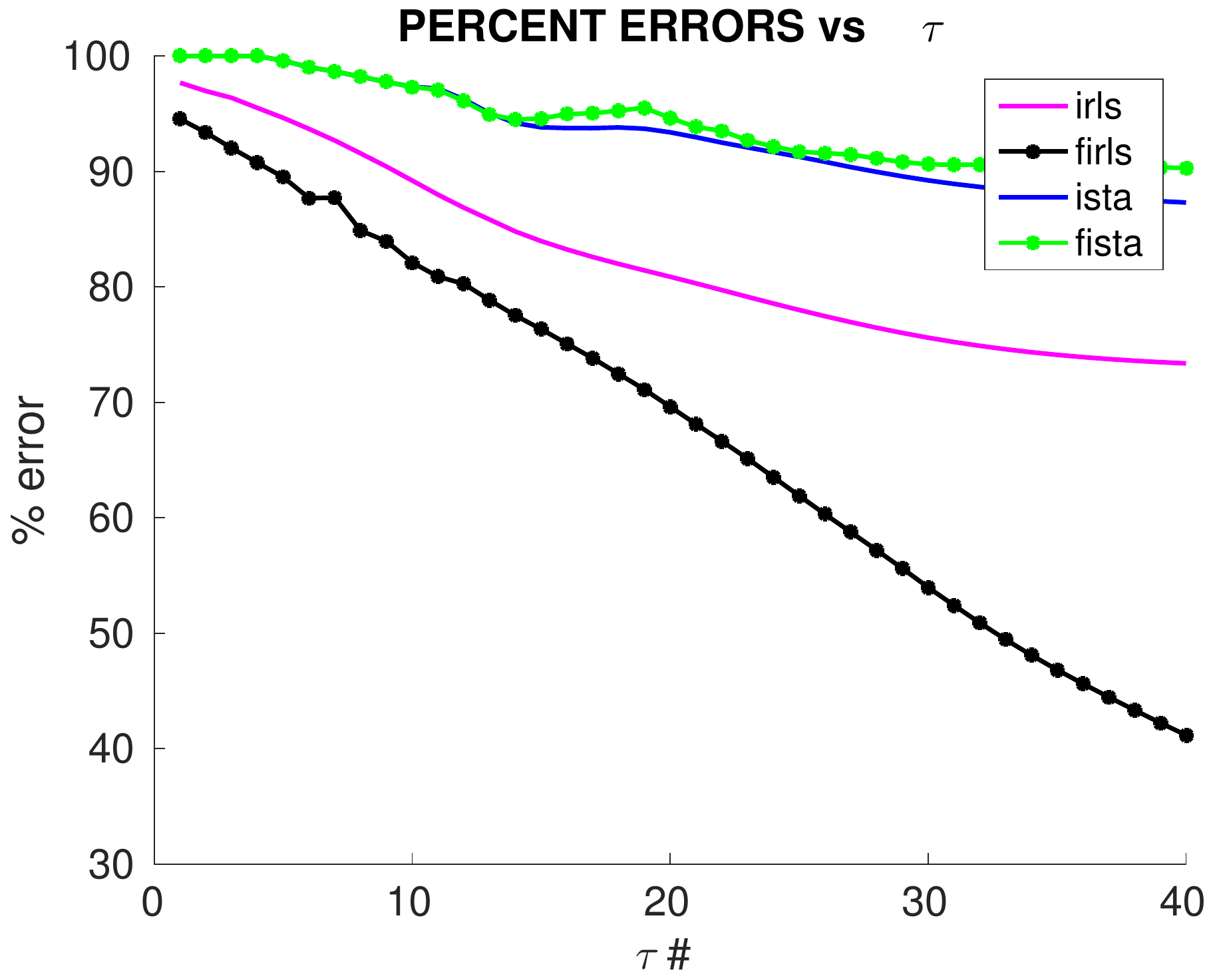}
}
\centerline{
\includegraphics[scale=0.16]{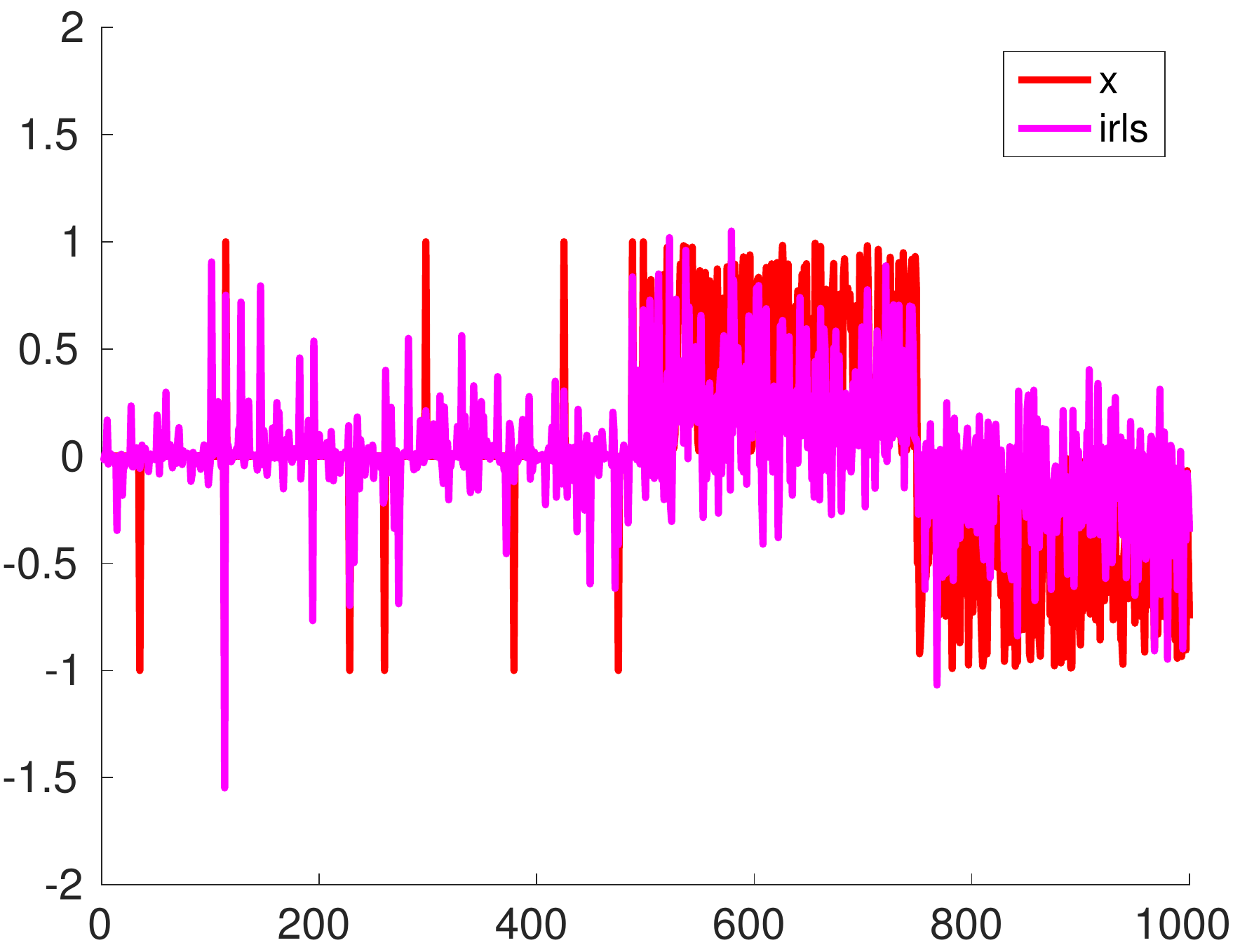}
\includegraphics[scale=0.16]{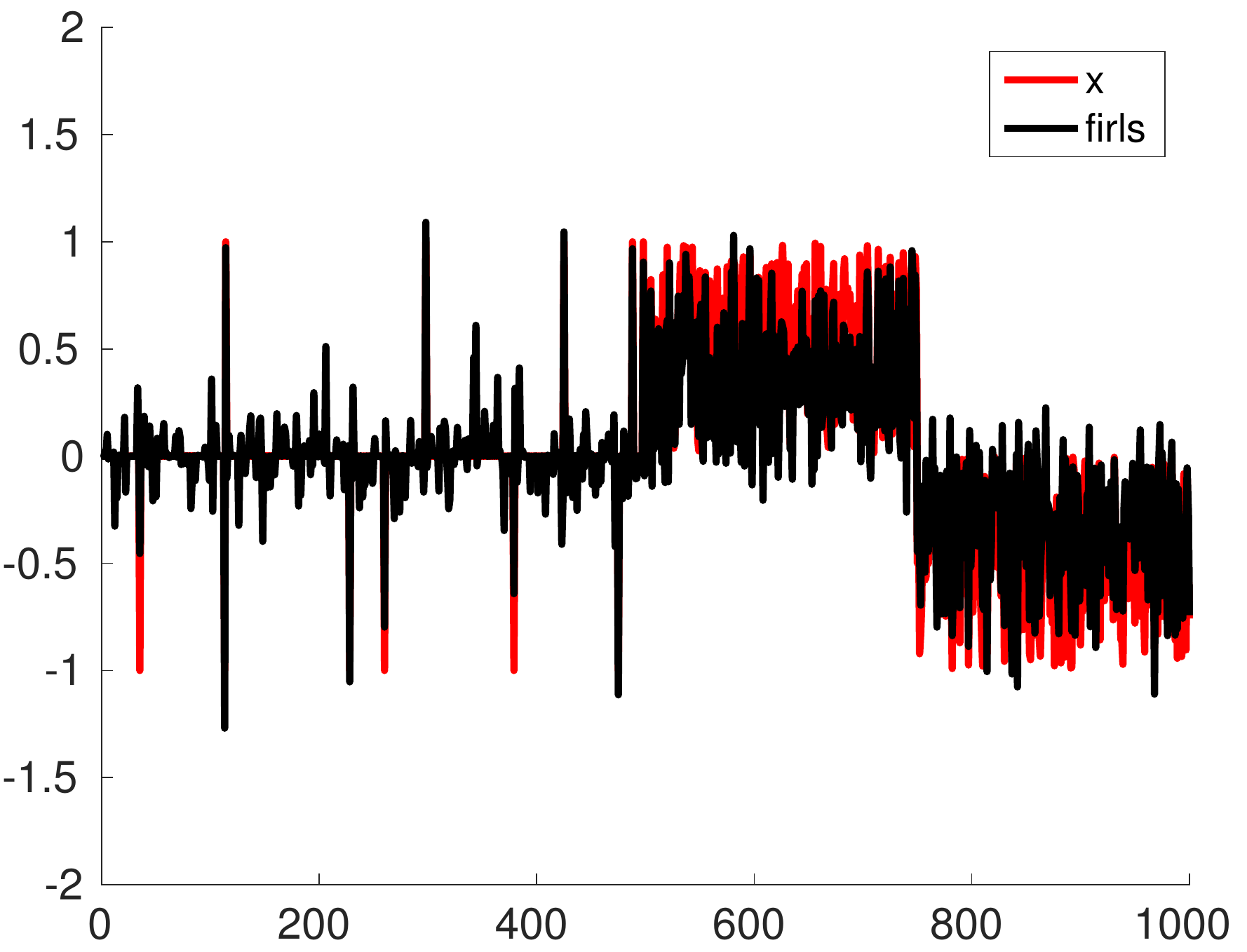}
\includegraphics[scale=0.16]{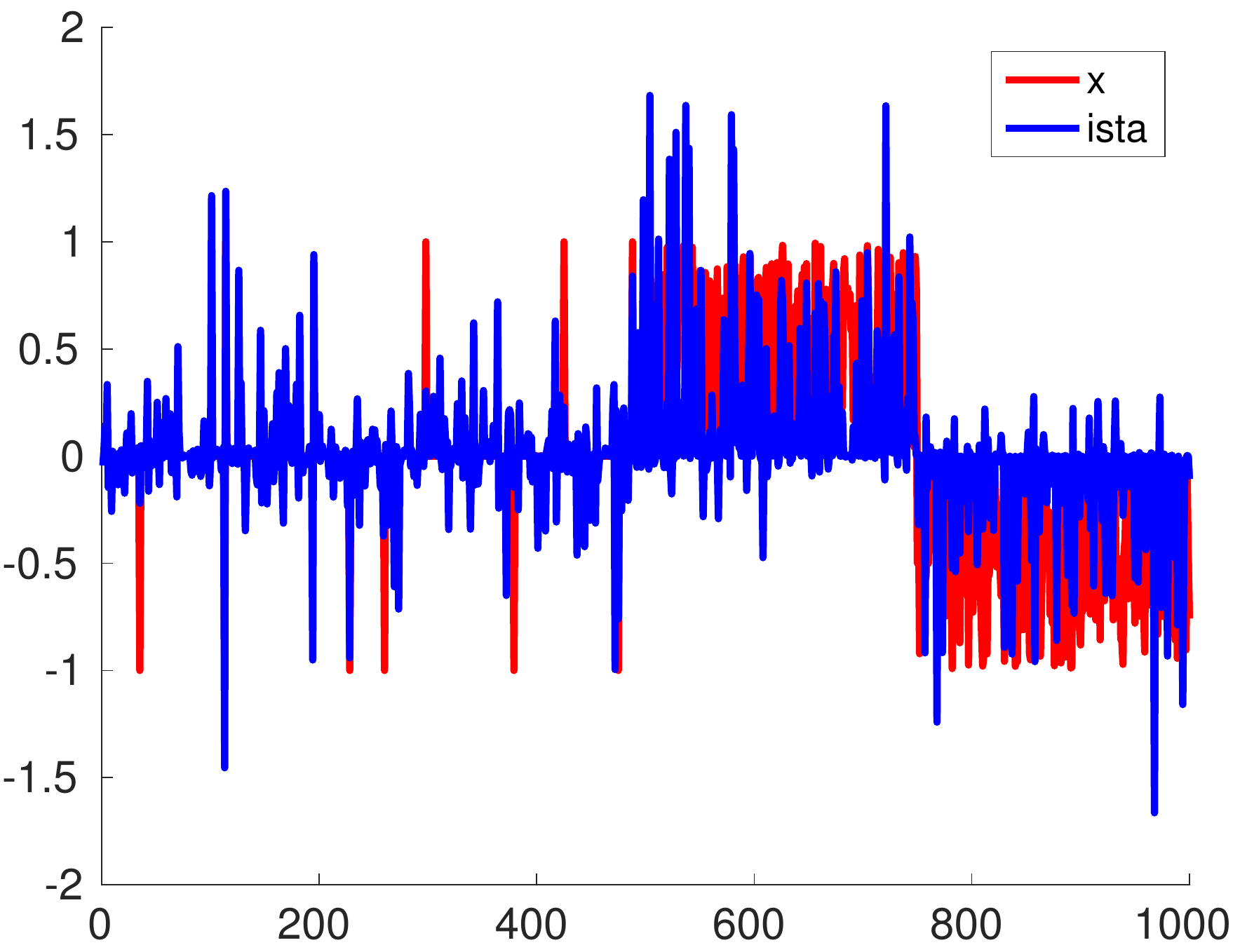}
\includegraphics[scale=0.16]{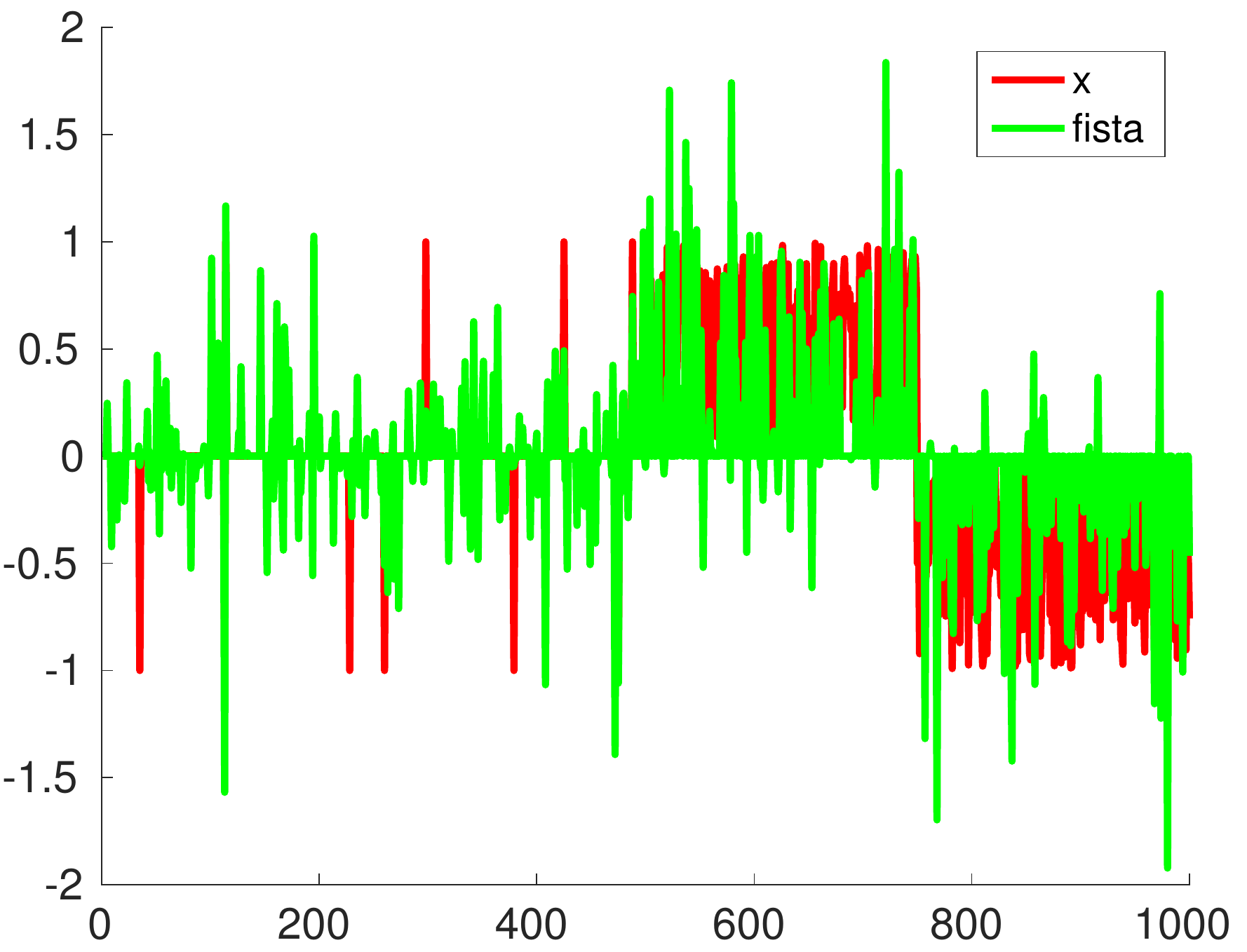}
}
\caption{Row 1: half sparse / half dense model $x$ and the recovery percent 
errors vs $\tau$. 
Row 2: final recovered solution with algorithms IRLS, FIRLS, ISTA, FISTA vs $x$.}
\label{fig:numerics3}
\end{figure*}

\newpage

\begin{figure*}[ht!]
\centerline{
\includegraphics[scale=0.6]{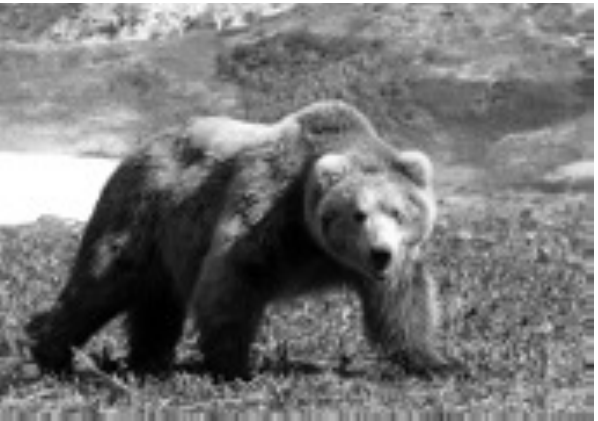}
\includegraphics[scale=0.6]{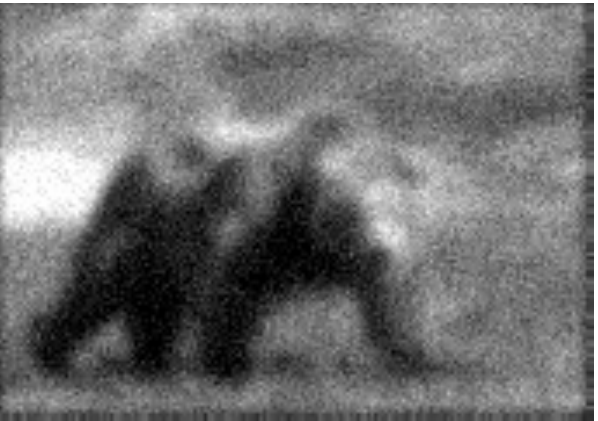}
\includegraphics[scale=0.6]{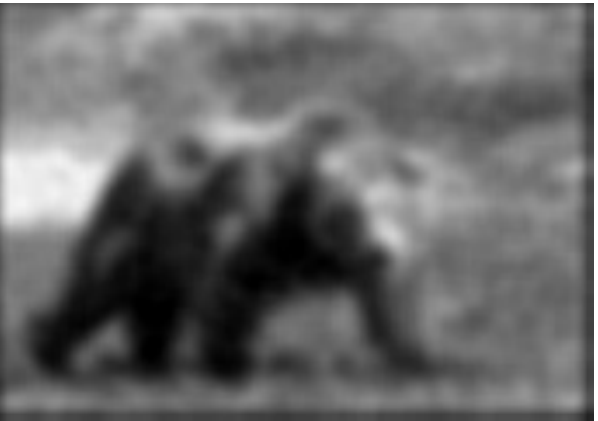}
\includegraphics[scale=0.6]{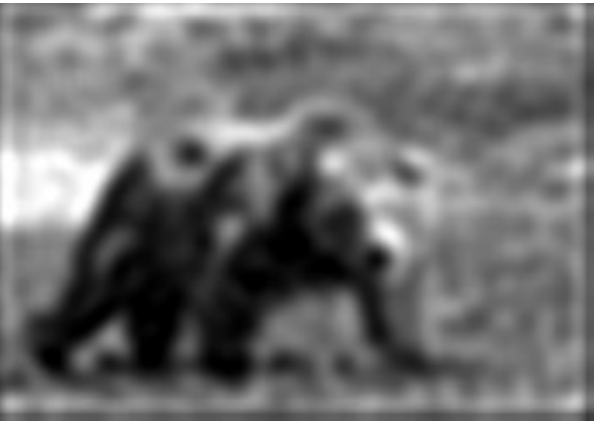}
}
\centerline{
\includegraphics[scale=0.6]{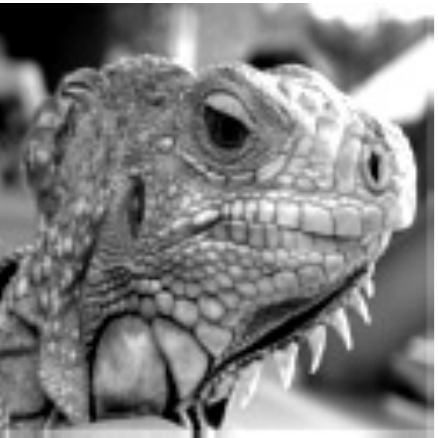}
\includegraphics[scale=0.6]{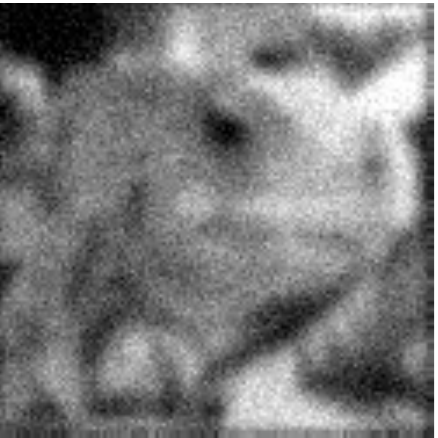}
\includegraphics[scale=0.6]{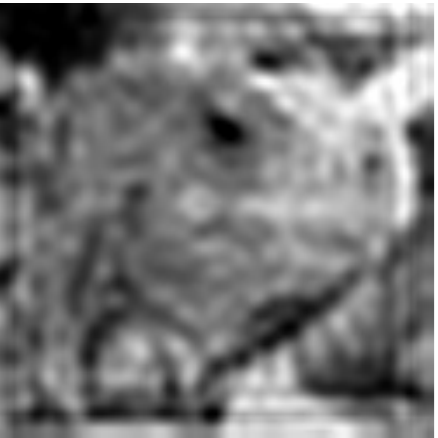}
\includegraphics[scale=0.6]{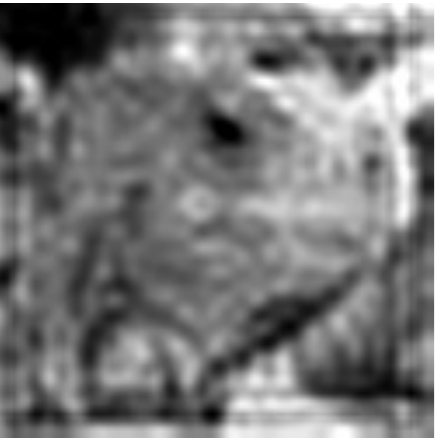}
}
\caption{Two examples of blurred and noisy image reconstruction with IRLS. In 
each row from left to right: original image, blurred and noisy image, and two 
reconstructions with IRLS using $q = 1$ and $q = 0.85$.\label{fig:numerics4}}
\end{figure*}

\vspace{3.mm}
\section{Conclusions}
This manuscript presents a new iterative algorithm for obtaining regularized 
solutions to least squares systems of equations with sparsity constraints. The proposed 
iteratively reweighted least squares algorithm extends the work of 
\cite{daubechies2010iteratively} and is similar in form to the popular ISTA and 
FISTA algorithms \cite{ingrid_thresholding1,Beck2009}; it has the added 
benefit of being able to minimize a more general sparsity promoting functional. 
The main contribution of this work is the analysis of the algorithm, relying on
matching the approximation rate to the original functional
of a smoothened surrogate functional to the speed of convergence of the iterates;
this methodology can likely also be applied to other situations. The presented IRLS 
algorithm \eqref{eq:irls_scheme} is very simple to implement and use; it offers 
performance similar to popular thresholding schemes, including the speedup benefit 
from the FISTA formulation. Because the surrogate functionals are all quadratic in the 
$x_k$, they lend themselves naturally to the use of a conjugate gradient approach, 
which enables further speed-up, as shown elsewhere 
\cite{sv_thesis,2015arXiv150904063F}.

\vspace{5.mm}
\bibliographystyle{plain}
\bibliography{master}

\end{document}